\documentclass{article}

\usepackage[english]{babel}
\usepackage{amsfonts,dsfont,mathrsfs}
\usepackage{amsmath,amssymb,amsthm,amsfonts}
\usepackage{graphicx}
\usepackage{caption}
\usepackage{xcolor} 
\usepackage{makeidx}
\usepackage{todonotes}
\usepackage{epsfig}
\usepackage{graphics}
\usepackage{float}
\usepackage{multirow}
\usepackage{color}
\usepackage{fullpage}
\usepackage[normalem]{ulem} 
\usepackage{xspace}
\usepackage{enumitem,graphicx}
\usepackage[colorlinks=true]{hyperref}
\usepackage{pgfplots}
\usepackage{bbm}
\pgfplotsset{compat=1.18}
\usepackage{subfigure}

\usepackage{authblk}
\makeindex

\makeatletter

\newtheorem{Lemma}{Lemma}[section]

\newtheorem{theorem}{Theorem}[section]

\newtheorem{proposition}{Proposition}[section]

\newtheorem{hypothesis}{Hypothesis}

\theoremstyle{definition}
\newtheorem{example}{Example}[section]
\newtheorem{remark}{Remark}

\definecolor{darkred}{rgb}{1, 0.1, 0.3}
\definecolor{darkblue}{rgb}{0.1, 0.1, 1}
\definecolor{darkgreen}{rgb}{0,0.6,0.5}

\newcommand{\R}{\mathbb{R}}
\renewcommand{\H}{\mathcal{H}}

\newcommand{\paren}[1]{\left(#1\right)}

\newcommand{\HH}{{\mathcal{H}}}

\DeclareMathOperator{\argmin}{argmin}
\DeclareMathOperator{\argmax}{argmax}

\let\epsilon\varepsilon

\begin{document}
	
	\title{Penalty-Based Smoothing and Accelerated Inertial Dynamics for Convex Supremum Optimization}

	\author[1]{Samir Adly\thanks{\texttt{samir.adly@unilim.fr}}}
	\author[2,3]{Juan Jos\'e Maul\'en\thanks{\texttt{juan.maulen@postdoc.uoh.cl}}}
	\author[2,3]{Emilio Vilches\thanks{\texttt{emilio.vilches@uoh.cl}}}
	
	\affil[1]{XLIM Laboratory, University of Limoges, 123 Avenue Albert Thomas, 87060 Limoges CEDEX, France}
	\affil[2]{Instituto de Ciencias de la Ingeniería, Universidad de O’Higgins, Rancagua, Chile}
	\affil[3]{Centro de Modelamiento Matem\'atico (CNRS IRL2807),
		Universidad de Chile, Beauchef 851, Santiago, Chile}

	\date{\today}

	\maketitle
	
	\begin{abstract}
		We study dual-penalty smoothing for convex functions with a finite-dimensional supremum representation. The regularization yields a differentiable surrogate with controlled approximation error, a single-valued dual maximizer, and explicit gradient formulas. We then study an accelerated inertial dynamic with vanishing damping driven by a time-dependent regularized function whose parameter decreases to zero. The global well-posedness of this system is not an immediate consequence of standard ODE theory, since the gradient of the regularized objective is in general only locally Lipschitz and its local Lipschitz modulus may deteriorate as the smoothing parameter vanishes. We therefore establish existence and uniqueness of global trajectories through a local well-posedness analysis and a continuation argument. Then, assuming that the original problem has a minimizer and under mild integrability conditions on the regularization schedule, we establish an accelerated $\mathcal{O}(t^{-2})$ decay estimate for the regularized and original objective values and, in the regime $\alpha>3$, a sharper $o(t^{-2})$ decay together with weak convergence of trajectories to a minimizer of the original nonsmooth problem via an Opial-type argument. Applications to multiobjective optimization (through Chebyshev/max scalarization) and to distributionally robust optimization (via entropic regularization over ambiguity sets) illustrate the scope of the framework.
	\end{abstract}

	\section{Introduction}
	Pointwise supremum functions and convex-concave saddle formulations are ubiquitous in contemporary optimization \cite{nemirovski2004prox,nesterov2005smooth,MR4279933}. They arise, among others, in \emph{min-max} and robust models \cite{MR2546839,MR2066239}, in scalarizations of multiobjective optimization \cite{MR1784937,ehrgott2005multicriteria}, and in distributionally robust optimization (DRO), where one minimizes a worst-case expectation over an ambiguity set \cite{MR4362585}.  A recurring algorithmic difficulty is that supremum functions are typically nonsmooth even when the underlying data are smooth, which prevents the direct use of fast first-order techniques and complicates the analysis of accelerated dynamics.
	
	In this paper, we study a smoothing mechanism and the associated accelerated second-order inertial continuous-time dynamics for the problem 
	\begin{equation}\label{eq:intro_prob}
		\min_{x\in\mathcal{H}} \varphi(x):=\sup_{\lambda\in Q}\sum_{i=1}^m \lambda_i g_i(x),
	\end{equation}
	where $\mathcal{H}$ is a real Hilbert space, $g_i \colon \mathcal{H}\to\mathbb{R}$ are convex functions, and $Q\subset\mathbb{R}^m_+$ is a nonempty compact convex set. The model \eqref{eq:intro_prob} is sufficiently general to encompass, among other examples, pointwise maxima of finitely many convex functions, norm- and gauge-type penalties, box-weighted hinge-type aggregates, and distributionally robust optimization problems in which $Q$ represents an ambiguity set over a finite family of scenarios (see the examples below). Despite this broad scope, the supremum structure in \eqref{eq:intro_prob} gives rise to two closely related analytical difficulties. First, the resulting objective is generally nonsmooth, and the maximizing parameter need not be unique, while the associated argmax mapping may fail to be continuous, which complicates the variational analysis of the problem, including subdifferential calculus, stability, and sensitivity analysis, and may also hinder the design of efficient numerical methods. Second, the smoothing parameter affects the regularity of the surrogate: its gradient Lipschitz bound may grow as $\mu(t) \downarrow 0$, and the inertial analysis must account for the variation of $\mu(t)$.
	
	Our aim is to use the supremum representation to obtain an explicit smooth approximation and convergence guarantees for the associated continuous-time inertial system.
	
	For the supremum model \eqref{eq:intro_prob}, we use dual-penalty smoothing. A strongly convex penalty in the variable $\lambda$ selects a unique maximizer $\lambda^\mu(x)$ and yields a differentiable approximation $\varphi_\mu$. Under the assumptions stated below, its gradient is Lipschitz continuous on bounded subsets of $\mathcal H$.
	
	On the primal side, the surrogate defines a gradient flow and the inertial system studied below. Under the stated assumptions, letting $\mu(t)\downarrow0$ links their trajectories to minimizers of the original supremum objective.
	
	Since max-type scalarizations of vector problems fit directly into \eqref{eq:intro_prob}, it is natural to compare our results with the rapidly developing dynamical-systems literature for multiobjective optimization.
	Sonntag and Peitz introduced an inertial multiobjective gradient-like system with \emph{asymptotic vanishing damping} and proved (in the smooth convex setting) existence of solutions in finite dimensions, weak convergence of bounded trajectories to weakly Pareto optimal points, and an $\mathcal O(t^{-2})$ rate for function values measured via a merit function \cite{sonntag2024fast}.
	In a complementary direction, they derived efficient discretizations and Nesterov-type accelerated algorithms for smooth convex multiobjective problems in Hilbert spaces, with weak convergence guarantees for the proposed inertial schemes and practical accelerated variants that reduce per-iteration subproblem costs \cite{sonntagpeitz2024fast}.
	More recently, Bo{\c t} and Sonntag proposed a second-order system that combines asymptotically vanishing damping with \emph{vanishing Tikhonov regularization applied componentwise}, establishing fast merit-function convergence and, in certain regimes, strong convergence (selection) toward a weak Pareto optimal point; this yields a mechanism to enforce strong convergence by regularization in the multiobjective setting \cite{BotSonntag2026JMAA}.
	Finally, Sonntag, Gebken, M{\"u}ller, Peitz, and Volkwein extended an efficient descent method for \emph{nonsmooth} multiobjective problems to general Hilbert spaces by computing approximations of the Clarke subdifferential and proving that every accumulation point of the generated sequence is Pareto critical under standard assumptions \cite{SonntagEtAl2024Nonsmooth}.
	
	Our contributions are complementary to these works in two main respects. First, we consider the supremum class \eqref{eq:intro_prob} with a finite-dimensional dual set $Q$, which includes max-type scalarizations in multiobjective optimization as well as finite-scenario robust and distributionally robust models. Second, instead of approximating generalized derivatives (e.g., Clarke-type constructions), we exploit the supremum structure to build smooth surrogates through dual regularization, and we study the resulting first-order and inertial dynamics.
	% First, we focus on the broader supremum/saddle class \eqref{eq:intro_prob}, which includes max-type scalarizations in multiobjective optimization but also robust and distributionally robust models, and can accommodate (in principle) continuous or infinite index sets through suitable convex-concave saddle representations. Second, instead of handling nonsmoothness via direct approximations of generalized derivatives (e.g., Clarke-type constructions), we exploit the supremum/saddle structure to build smooth surrogates through dual regularization, and we study the resulting primal (and, when convenient, primal-dual) first-order and inertial dynamics. This provides a unified route to convergence analysis that naturally interfaces with accelerated vanishing-damping dynamics (cf.~\cite{sonntag2024fast,sonntagpeitz2024fast}) and with trajectory selection through vanishing regularization mechanisms (cf.~\cite{BotSonntag2026JMAA}), while remaining tailored to the supremum functions that are intrinsic to robust and distributionally robust optimization.
	
	The present analysis concerns the finite-dimensional dual supremum model \eqref{eq:intro_prob}, which includes max scalarizations and finite-scenario robust objectives. The inertial system studied here is driven by one scalarized function. In the multiobjective application, its limiting point minimizes that scalarization and is therefore weakly Pareto optimal. The conclusion is tied to the chosen scalarization.
	
	Following the general philosophy of Nesterov-type smoothing \cite{nesterov2005smooth}, we introduce for $\mu>0$ the penalty-regularized function
	\begin{equation*}
		\varphi_\mu(x)
		:=\max_{\lambda\in Q}\left\{\sum_{i=1}^m \lambda_i g_i(x)-\mu\mathcal{D}(\lambda)\right\},
	\end{equation*}
	where $\mathcal{D}$ is nonnegative and $\sigma$-strongly convex on $Q$ with $\sup_{\lambda \in Q}\mathcal{D}(\lambda)<+\infty$.
	This yields (i) a \emph{controlled approximation from below} $\varphi_\mu\uparrow\varphi$ as $\mu\downarrow 0$ with an explicit uniform error bound of order $\mu$, and (ii) under mild assumptions a \emph{smooth} convex surrogate with locally Lipschitz gradient, admitting an envelope representation via the unique maximizer $\lambda^\mu(x)$.
	The framework covers classical smooth approximations of $\max$ functions such as log-sum-exp (entropic/KL penalty) and quadratic proximal penalties, and it makes explicit the characteristic $1/\mu$ degradation of the Lipschitz modulus of $\nabla\varphi_\mu$ as $\mu\downarrow 0$ \cite{nesterov2005smooth}, a key point for both complexity estimates and dynamical analyses. 
	
	The idea of regularizing the inner maximization problem to obtain smooth approximations of supremum-type functions was explored well before \cite{nesterov2005smooth}, which popularized it in the context of complexity estimates. It goes back to work on nondifferentiable optimization \cite{bertsekas1975nondifferentiable}, was formulated in a penalty-based framework in \cite{bentalteboulle1989smoothing} and was unified with first-order methods in \cite{beck2012smoothing}. The asymptotic behavior of the associated penalty trajectories as the parameter vanishes was studied in \cite{cominetti1994asymptotic}, while schemes in which the smoothing parameter is driven to zero along iterations, a discrete counterpart of the schedule $\mu(t)$ considered here, appear for instance, in \cite{bot2015variable}. Related constructions also arise outside optimization, in discrete choice models \cite{zachary1978improved}.
	
	The proposed regularization scheme enables the use of first- and second-order dynamical systems. In this paper, we focus on second-order inertial dynamics, as they capture accelerated regimes and are closely connected to Nesterov-type methods. For completeness, the analysis of the corresponding first-order (gradient-flow) dynamics is provided in Appendix~\ref{gradient}.
	
	We apply this penalty-based smoothing to an inertial system for \eqref{eq:intro_prob}. Inertial systems with vanishing damping of the form
	\[
	\ddot{x}(t)+\frac{\alpha}{t}\dot{x}(t)+\nabla \Phi(x(t))=0
	\]
	provide a natural continuous-time counterpart of accelerated first-order optimization methods and arise, in particular, as ODE limits of Nesterov's accelerated scheme \cite{su2016differential}. In the smooth convex setting, these dynamics have been extensively studied. They exhibit the sharp function-value decay rate $\mathcal{O}(t^{-2})$, while for $\alpha>3$ the trajectories converge weakly to a minimizer under appropriate assumptions \cite{attouch2018fast}. 
	% Qu and Bian~\cite{MR4462746} already study an inertial equation with vanishing damping driven by a time-dependent smoothing of a nonsmooth convex objective. Under their smoothing and schedule assumptions, which include a global $L/\mu$ Lipschitz bound for the gradient, they obtain global trajectories, $\mathcal O(t^{-2})$ and $o(t^{-2})$ value rates, and weak convergence when $\alpha>3$. Here the smoothing family is constructed explicitly by penalizing the dual maximization in \eqref{eq:intro_prob}. 
	%Their well-posedness argument \cite[Proposition~2.1]{MR4462746} relies on
	% \cite[Definition~2.1(v)]{MR4462746}, which requires the gradient of the smoothing
	% to be Lipschitz on the whole space with constant $L\mu^{-1}$. The dual-penalty family
	% considered here satisfies the remaining requirements of that definition (with
	% $\kappa=\mathcal C$), but in general its gradient is only Lipschitz on bounded sets
	% (Proposition~\ref{p:Lipschitz}), and this cannot be improved: Remark~\ref{r:qu_bian_counterex}
	% exhibits a maximum of two convex quadratics with globally Lipschitz gradients for which
	% $\nabla\varphi_\mu$ is not globally Lipschitz for any $\mu>0$. Hence the results of
	% \cite{MR4462746} cannot be invoked directly in such instances. We establish global
	% well-posedness under Hypothesis~\ref{hip:general}, where only local gradient bounds are
	% available, by combining local existence with an energy-based continuation argument; the
	% asymptotic analysis then follows energy arguments in the spirit of
	% \cite{MR4462746,attouch2018fast}.  
	Qu and Bian~\cite{MR4462746} study this equation for an abstract smoothing function of a nonsmooth convex objective. Under \cite[Definition~2.1]{MR4462746} and the schedule condition $\int_{t_0}^\infty t\mu(t)\,dt<\infty$ with $\mu$ decreasing, they prove global well-posedness \cite[Proposition~2.1]{MR4462746}, $\mathcal O(t^{-2})$ and, for $\alpha>3$, $o(t^{-2})$ value rates \cite[Theorem~3.1]{MR4462746}, and weak convergence \cite[Theorem~3.2]{MR4462746}. The dual-penalty family introduced in this work satisfies items (i)--(iv) and (vi) of \cite[Definition~2.1]{MR4462746} with $\kappa=\mathcal C$, and Hypothesis~\ref{as:mu_cont} coincides with their schedule assumption. Their asymptotic arguments use item (v) only through well-posedness. Hence, once global trajectories are available, the conclusions of Theorems~\ref{t:rates_inertial} and~\ref{t:traj} on values and weak convergence can also be derived from \cite{MR4462746}. We give self-contained proofs based on the identity $\partial_\mu\varphi_\mu=-\mathcal D(\lambda^\mu)$, which makes the variation of $\mu$ dissipative, and we add the velocity estimate $\|\dot x(t)\|=o(1/t)$. The difference lies in well-posedness. Item (v), a global Lipschitz bound $L\mu^{-1}$ on the gradient, fails for the dual-penalty family in general (see Remark~\ref{r:qu_bian_counterex}). When the data gradients grow superlinearly, which Hypothesis~\ref{hip:general} allows, $\nabla\varphi_\mu$ has neither a global Lipschitz bound nor linear growth, and global existence requires an a priori estimate. Theorem~\ref{teo:global_existence} provides it through an energy bound on bounded time intervals. It holds for every $\alpha>0$ and requires neither monotonicity of $\mu$, nor $\mu(t)\to0$, nor the existence of a minimizer.
	
	More precisely, we consider the nonautonomous inertial system
	\begin{equation}\label{eq:intro_inertial}
		\ddot{x}(t)+\frac{\alpha}{t}\dot{x}(t)
		+\nabla \varphi_{\mu(t)}(x(t))=0,
	\end{equation}
	with $\alpha>0$ and $\mu(t)\downarrow 0$ as $t\to\infty$. The time dependence introduces a genuine analytical difficulty: the natural Lyapunov functional inherits additional terms stemming from $\dot{\mu}(t)$, and the reference minimizer $x^{\ast} \in\operatorname{argmin}\varphi$ is, in general, \emph{not} a minimizer of $\varphi_{\mu(t)}$ for fixed $t$. Assuming that $\varphi$ attains its minimum, and under monotonicity of $\mu$ and the condition $\int_{t_0}^{\infty}t\mu(t)\,dt<\infty$, we obtain the rates and weak convergence stated below for this dual-penalty smoothing family.
	% Under monotonicity of $\mu$ and the condition $\int_{t_0}^{\infty}t\mu(t)\,dt<\infty$, we obtain the rates and weak convergence stated below for this dual-penalty smoothing family.
	%Its gradient is assumed to be Lipschitz continuous only on bounded subsets of $\mathcal H$. When this penalty family also satisfies the global gradient bound of \cite{MR4462746}, their convergence theorems apply directly. 
	
	The main contributions of this paper can be summarized as follows:
	\begin{itemize}[leftmargin=2em]
		\item For the finite-dimensional dual supremum model \eqref{eq:intro_prob}, we verify the uniform approximation estimate, uniqueness of the penalized maximizer, and explicit formulas for $\nabla\varphi_\mu$ and $\partial_\mu\varphi_\mu$. We also identify the limit of the dual maximizer as $\mu\downarrow0$ and give examples based on max functions, norm aggregates, and distributionally robust optimization.
		
		\item We derive local Lipschitz estimates for $\nabla\varphi_\mu$ that make explicit the tradeoff between smoothness and approximation accuracy, including the characteristic $1/\mu$ scaling of Nesterov-type smoothing \cite{nesterov2005smooth}.
		
		\item We establish existence and uniqueness of solutions to the nonautonomous inertial system \eqref{eq:intro_inertial}. Since the available Lipschitz estimates are only local in the state variable and depend on the time-varying smoothing parameter, global well-posedness does not follow directly from standard ODE arguments. We first obtain local well-posedness and suitable a priori estimates, and then use a continuation argument to extend the solution globally in time. This result requires neither global Lipschitz nor linear-growth bounds on $\nabla\varphi_\mu$, nor monotonicity of $\mu$, nor the existence of minimizers.
		
		\item For the inertial system \eqref{eq:intro_inertial}, we adapt the energy estimate to the one-sided penalty approximation and prove, under explicit conditions on $\mu(t)$ and assuming $\argmin\varphi\neq\emptyset$, an $\mathcal O(t^{-2})$ estimate for $|\varphi_{\mu(t)}(x(t))-\inf\varphi|$ and $\varphi(x(t))-\inf\varphi$. For $\alpha>3$, we obtain $o(t^{-2})$ decay for both gaps and weak convergence to a minimizer of $\varphi$. These conclusions coincide with those of Qu and Bian~\cite[Theorems~3.1 and~3.2]{MR4462746} under the same schedule assumption; here they hold for a smoothing family that does not satisfy their global Lipschitz requirement \cite[Definition~2.1(v)]{MR4462746}.
		
		\item We instantiate the theory in two application domains. The first concerns multiobjective optimization through Chebyshev/max scalarization, where the limit of the trajectory minimizes the chosen scalarization and is therefore weakly Pareto optimal; we position this approach with respect to accelerated multiobjective methods \cite{sonntag2024fast,sonntagpeitz2024fast}. The second concerns distributionally robust optimization with KL-type regularization over moment-based ambiguity sets, leading to tractable regularized objectives, with numerical experiments for several values of $\alpha$.
	\end{itemize}
	
	The paper is organized as follows. In Section~\ref{sec:prelim}, we introduce the mathematical preliminaries and develop the penalty-based smoothing framework for supremum functions, including its approximation and differentiability properties, together with several representative examples. Section~\ref{sec:dynamics} is devoted to the analysis of the inertial dynamical system with vanishing damping driven by the time-dependent regularized function. We establish its well-posedness through a local analysis and a continuation argument, and prove an $\mathcal{O}(t^{-2})$ decay estimate for the regularized function values. For $\alpha>3$, we further obtain additional integrability properties and weak convergence of the trajectories.  Section~\ref{eq:appl} presents two application domains, namely multiobjective optimization and distributionally robust optimization (DRO), together with numerical experiments illustrating the predicted convergence behavior. The paper concludes with some remarks and directions for future research.

	\section{Mathematical Preliminaries and the Penalty Framework}\label{sec:prelim}

	\subsection{Mathematical preliminaries and notation}
	
	Throughout the paper, $\mathcal H$ denotes a real Hilbert space endowed with inner product
	$\langle\cdot,\cdot\rangle$ and associated norm $\|x\|:=\sqrt{\langle x,x\rangle}$.
	When $\mathcal H=\mathbb R^n$, we also use the $p$-norm
	$\|x\|_p:=\big(\sum_{i=1}^n |x_i|^p\big)^{1/p}$ for $p\ge 1$, with $\|x\|_\infty:=\max_{1\le i\le n}|x_i|$.
	For $m\in\mathbb N$, we denote by $\mathbbm 1\in\mathbb R^m$ the all-ones vector and by $\Delta_m:=\left\{p\in\mathbb R^m_+:\ \langle \mathbbm 1,p\rangle=1\right\}$ the $m$-dimensional probability simplex. For a nonempty set $Q\subset\mathcal H$, we define the distance function
	$\operatorname{dist}_Q(x):=\inf_{y\in Q}\|x-y\|$; if, in addition, $Q$ is closed and convex, we write
	$\operatorname{proj}_Q(x):=\arg\min_{y\in Q}\|x-y\|$ for the (single-valued) metric projection, so that
	$\operatorname{dist}_Q(x)=\|x-\operatorname{proj}_Q(x)\|$.
	
	For a proper function $f\colon \mathcal H\to\mathbb R\cup\{+\infty\}$, we use the indicator function of a set $Q$,
	$\iota_Q(x)=0$ if $x\in Q$ and $\iota_Q(x)=+\infty$ otherwise, and the Fenchel conjugate
	$f^\ast(y):=\sup_{x\in\mathcal H}\{\langle y,x\rangle-f(x)\}$.
	When $f$ is convex, its (convex-analytic) subdifferential at $x$ is
	\[
	\partial f(x):=\{v\in\mathcal H \colon \ f(y)\ge f(x)+\langle v,y-x\rangle\ \ \forall y\in\mathcal H\}.
	\]
	For a closed convex set $Q\subset\mathcal H$, we denote by $N_Q(x)$ the normal cone,
	\[
	N_Q(x):=\{v\in\mathcal H \colon \ \langle v,y-x\rangle\le 0\ \ \forall y\in Q\}\quad (x\in Q),
	\qquad N_Q(x):=\emptyset\quad (x\notin Q),
	\]
	and by $\operatorname{ri}(Q)$ its relative interior. We say that a function $\mathcal D \colon Q\to\mathbb R$ is
	$\sigma$-strongly convex on $Q$ if, for all $x,y\in Q$ and $\theta\in[0,1]$,
	\[
	\mathcal D(\theta x+(1-\theta)y)\le \theta \mathcal D(x)+(1-\theta)\mathcal D(y)
	-\frac{\sigma}{2}\theta(1-\theta)\|x-y\|^2.
	\]
	For differentiable $h \colon \mathcal H\to\mathbb R$ and a set $B\subset\mathcal H$, we say that $\nabla h$ is
	$L$-Lipschitz on $B$ if $\|\nabla h(x)-\nabla h(y)\|\le L\|x-y\|$ for all $x,y\in B$.\\
	\noindent When using entropic penalties on $\Delta_m$, we adopt the convention $0\log 0:=0$ and write, for a probability distribution $u=(u_1,\ldots,u_m)$,
	\[
	\operatorname{KL}(p\|u):=\sum_{i=1}^m p_i\log\!\Big(\frac{p_i}{u_i}\Big).
	\]
	Finally, for functions of time we use the standard asymptotic notation: $a(t)=\mathcal O(t^{-2})$ means that
	$|a(t)|\le C/t^2$ for all large $t$, while $a(t)=o(t^{-2})$ means $t^2 a(t)\to 0$ as $t\to\infty$; we also use
	$L^1(t_0,\infty)$ for Lebesgue integrability on $(t_0,\infty)$ and the notation $x_n\rightharpoonup x$ for weak
	convergence in $\mathcal H$. For completeness, we recall in Appendix~\ref{opial_Appendix} the continuous-time version of Opial's lemma (see \cite{opial1967weak}), which we use to conclude weak convergence of the trajectories.
	
	\subsection{Penalty Framework}
	Let $g_1,\ldots, g_m\colon \mathcal{H} \to \mathbb{R}$ be convex functions, and let $Q \subset \R^m_+$ be nonempty, convex, and compact. Let us consider the optimization problem
	\begin{equation}\label{eq:prob_max}
		\min_{x \in \mathcal{H}}\varphi(x):=\max_{\lambda \in Q} \left(  \sum_{i=1}^{m}\lambda_i g_i(x)\right).
	\end{equation} 
	We follow Nesterov's smoothing approach \cite{nesterov2005smooth} and define
	\begin{equation}\label{eq:Framework}
		\varphi_{\mu}(x):=\max_{\lambda\in Q}\left\{\sum_{i=1}^m \lambda_i g_i(x)-\mu \mathcal{D}(\lambda)\right\}.
	\end{equation}
	Here $\mu$ is a positive parameter and $\mathcal{D}(\cdot)$ is a nonnegative convex penalty function defined on $Q$ with $\mathcal{C}:=\sup_{\lambda \in Q}\mathcal{D}(\lambda)<\infty$.  We assume that $\mathcal{D}$ is lower semicontinuous on $Q$. This condition ensures that the maximum in \eqref{eq:Framework} is attained. When required, we will adopt the notation $g(x) = (g_1(x),\ldots,g_m(x))$. Notice that $\varphi$ is a convex function, and when required, we will assume that attains its minimum. In this work, we will deal with the minimization problem over $\varphi(x)$ by using the auxiliary function $\varphi_\mu(x)$. We will refer to this function usually as the \textit{regularization} of $\varphi$. The formulation of problem~\eqref{eq:prob_max}, together with the construction of the regularizer, covers a broad range of optimization problems (see Section~\ref{eq:appl} and the Appendix \ref{Ejemplos-1000}). 
	
	The following proposition collects the main properties of the proposed smoothing scheme. In particular, it quantifies the approximation error, establishes uniqueness of the regularized dual maximizer, and provides explicit differentiation formulas with respect to both the primal variable and the smoothing parameter.
	\begin{proposition}\label{p:properties}
		Let $\mathcal{H}$ be a real Hilbert space. Assume $g_1,\ldots, g_m\colon \mathcal{H} \to \mathbb{R}$ are convex, $Q\subset \mathbb{R}^m_+$ is a nonempty, compact and convex set, and $\mathcal{D}\colon Q \to \mathbb{R}$ is a $\sigma$-strongly convex function, lower semicontinuous and nonnegative on $Q$ and $\mathcal{C}:=\sup_{\lambda\in Q}\mathcal{D}(\lambda)<+\infty$. Then, $\varphi_\mu$ satisfies the following properties
		\begin{enumerate}[label={\rm (\roman*)}]
			\item For all $\mu>0$ and all $x\in \mathcal{H}$, 
			\begin{equation}\label{eq:prop1}
				0\leq \varphi(x)-\varphi_{\mu}(x)\leq \mathcal{C}\mu. 
			\end{equation}
			\item For all $\mu>0$, the map $x\mapsto \varphi_{\mu}(x)$ is convex. Moreover, the maximizer 
			\begin{equation}\label{eq:def_lambda_mu}
				\lambda^{\mu}(x)=(\lambda_1^{\mu}(x),\ldots,\lambda_m^{\mu}(x))\in \displaystyle{\argmax_{\lambda\in Q}}\left\{\sum_{i=1}^m \lambda_i g_i(x)-\mu \mathcal{D}(\lambda)\right\} 
			\end{equation}
			is unique. If, in addition, each $g_i$ is Fr\'echet differentiable on $\mathcal{H}$, then $x\mapsto \varphi_{\mu}(x)$ is differentiable and
			\begin{equation*}
				\nabla \varphi_{\mu}(x)=\sum_{i=1}^m \lambda_{i}^{\mu}(x)\nabla g_i(x).
			\end{equation*}
			\item For every $x\in \mathcal{H}$, the map $\mu \mapsto \varphi_{\mu}(x)$ is differentiable and
			\begin{equation}\label{eq:prop3}
				\frac{d}{d\mu}\varphi_{\mu}(x)= \frac{1}{\mu} \left(\varphi_{\mu}(x)-\sum_{i=1}^m \lambda_i^{\mu}(x)g_i(x)\right)=-\mathcal{D}(\lambda^{\mu}(x))\leq 0.
			\end{equation}
		\end{enumerate}
	\end{proposition}
	\begin{proof} 
		Let $x\in\mathcal H$ and $\mu>0$. Since $Q$ is compact and $\lambda\mapsto \sum_{i=1}^m \lambda_i g_i(x)$ is continuous on $Q$, the maximum defining $\varphi(x)$ is attained. Moreover, since $\mathcal D$ is lower semicontinuous on $Q$, the function $\lambda\mapsto \sum_{i=1}^m \lambda_i g_i(x)-\mu\mathcal D(\lambda)$ is upper semicontinuous on $Q$. Hence, by compactness of $Q$, the maximum defining $\varphi_\mu(x)$ is also attained. Moreover, since $\mathcal D(\lambda)\ge 0$ for all $\lambda\in Q$,
		\[
		\sum_{i=1}^m \lambda_i g_i(x)-\mu\mathcal D(\lambda)\le \sum_{i=1}^m \lambda_i g_i(x),
		\]
		which implies that $\varphi_\mu(x)\le \varphi(x)$.
		On the other hand, let $\lambda^{\ast}(x)\in\arg\max_{\lambda\in Q}\sum_{i=1}^m\lambda_i g_i(x)$. Then
		\[
		\varphi_\mu(x)\ge \sum_{i=1}^m \lambda_i^{\ast}(x) g_i(x)-\mu\mathcal D(\lambda^*(x))
		= \varphi(x)-\mu\mathcal D(\lambda^*(x))
		\ge \varphi(x)-\mu\mathcal C,
		\]
		where $\mathcal C:=\sup_{\lambda\in Q}\mathcal D(\lambda)$. This proves (i). \\
		\noindent For (ii), for each fixed $\lambda\in Q$, the map
		$x\mapsto \sum_{i=1}^m \lambda_i g_i(x)-\mu\mathcal D(\lambda)$ is convex as a nonnegative combination of convex functions,
		hence $\varphi_\mu$ is convex as a pointwise maximum of convex functions.
		Moreover, since $\mathcal D$ is $\sigma$-strongly convex, the map
		$\lambda\mapsto \sum_{i=1}^m \lambda_i g_i(x)-\mu\mathcal D(\lambda)$ is $\mu\sigma$-strongly concave on $Q$. Since $Q$ is compact and the objective is upper semicontinuous in $\lambda$, a maximizer exists, and strong concavity yields its uniqueness. We denote it by $\lambda^\mu(x)$.
		If, in addition, each $g_i$ is Fr\'echet differentiable, then $\varphi_\mu$ is differentiable and
		\[
		\nabla\varphi_\mu(x)=\sum_{i=1}^m \lambda_i^\mu(x)\nabla g_i(x),
		\]
		by a standard envelope argument for parametric maximization, using the uniqueness of $\lambda^\mu(x)$.\\
		\noindent For (iii), fix $x\in\H$. The bound $0 \leq \mathcal{D}\leq \mathcal{C}$ shows that $\nu \mapsto \varphi_{\nu}(x)$ is locally Lipschitz. If $\nu \to \mu >0$, compactness and uniqueness give $\lambda^{\nu}(x)\to \lambda^{\mu}(x)$. Moreover, the identity 
		\[\nu\mathcal{D}(\lambda^{\nu}(x)) = \sum_{i=1}^m \lambda_{i}^{\nu} (x)g_i(x) - \varphi_{\nu}(x) \]
		implies $\mathcal{D}(\lambda^{\nu}(x)) \to \mathcal{D}(\lambda^{\mu}(x))$. Comparing the maxima at $\lambda^{\mu}(x)$ and $\lambda^{\nu}(x)$ now yields
		\[\dfrac{d}{d\mu}\varphi_{\mu}(x) = -\mathcal{D}\left( \lambda^{\mu}(x)\right)\leq 0.\]
	\end{proof}
	
	\begin{remark}\label{r:perspective} Let $\phi = \mathcal{D} + \iota_Q$, from which one has the following conjugate representation $\varphi_{\mu}(x) = \mu \phi^{\ast}\left(\frac{g(x)}{\mu}\right)$, where $\phi^{\ast}$ stands for the Fenchel conjugate of $\phi$. The use of the extended-value function $\phi$ will be convenient in the forthcoming analysis because it will avoid differentiability requirements on $\mathcal{D}$ at the boundary of $Q$.
	\end{remark}
	\begin{remark}\label{r:minimums} The property \eqref{eq:prop1} indicates that $\varphi_\mu$ is an approximation from below to $\varphi$, which satisfies $\varphi_\mu \to \varphi$ when $\mu \to 0$. Moreover, if $\argmin \varphi_\mu \neq \emptyset$,  we obtain that  $\min \varphi_\mu \leq \min \varphi \leq \min \varphi_\mu +  \mathcal{C}\mu$.
	\end{remark}

	The following result, proved in Appendix~\ref{prueba-Lipschitz}, shows that the stability of the dual maximizer yields the local gradient bound required for the differential equation.
	\begin{proposition}\label{p:Lipschitz} Assume the hypotheses of Proposition \ref{p:properties}, and suppose that each $g_i$ is Fréchet differentiable on $\H$. Let $B \subset \H$ be a nonempty, closed, bounded, convex set on which each $\nabla g_i$ is Lipschitz continuous with constant $L_i$. Then, for every $x,y \in B$ and every $\mu>0$, 
		\[\Vert \nabla\varphi_\mu(x) - \nabla\varphi_\mu(y) \Vert  \leq \left( M_Q L_g+\dfrac{G_B^2}{\mu\sigma } \right) \Vert x-y \Vert,\]    
		where $M_Q:= \sup_{\lambda \in Q} \Vert \lambda\Vert_1$, $L_g:= \sum_{i=1}^m L_i$, and  $G_B := \left(\sum_{i=1}^{m}\sup_{z \in B} \Vert \nabla g_i(z)\Vert^2\right)^{1/2}$.
	\end{proposition} 
	\begin{remark}\label{r:qu_bian_counterex} The local character of the bound in Proposition~\ref{p:Lipschitz} cannot be removed in general, and this is what prevents a direct use of \cite[Proposition~2.1]{MR4462746}. As an example, consider $\HH = \R^2$, $Q= \Delta_2$, $g_1(x) = (x_1 + x_2)^2$, $g_2(x) = (x_1 - x_2)^2$, so that $\varphi(x)= \max\left\lbrace g_1(x), g_2(x)\right\rbrace = \Vert x\Vert_1^2$. Consider the entropy penalty 
		\[\mathcal{D}(\lambda) = \sum_{i=1}^{2}\lambda_i \log(2\lambda_i) \quad\Longrightarrow \quad\varphi_\mu(x) = \mu \log \left(\frac{1}{2}\left( e^{g_1(x)/\mu} + e^{g_2(x)/\mu} \right) \right).\]
		Each $\nabla g_i$ is globally Lipschitz, and $\argmin \varphi = \lbrace 0\rbrace$. However, at $x = (s,0)$ one has 
		\[\frac{\partial^2 \varphi_\mu}{\partial x_2^2}(s,0) = 2 + \frac{4s^2}{\mu}.\]
		Notice that, for any fixed $\mu > 0 $, this second derivative grows unbounded as $s\to +\infty$. Thus, although each $\nabla g_i$ is globally Lipschitz, the Lipschitz constant of $\nabla \varphi_\mu$ is not globally bounded in $\R^2$, then the smoothing assumption in  \cite[Definition 2.1(v)]{MR4462746} fails, while all the results of the present work apply. In the previous example, $\nabla\varphi_\mu$ still grows linearly. This is no longer true when some $g_i$ grows faster than quadratically. For instance, let $Q=\Delta_m$ and $g_1(x)=\frac14\|x\|^4$. By \eqref{eq:prop1}, $\varphi_\mu\ge g_1-\mathcal C\mu$ grows quartically. A convex function whose gradient has linear growth grows at most quadratically. Hence, for every $\mu>0$, $\nabla\varphi_\mu$ is neither globally Lipschitz nor of linear growth, although Hypothesis~\ref{hip:general} holds.
	\end{remark}
	\noindent Another important point of the smoothing construction is the selection mechanism for the dual maximizer. 
	
	\begin{theorem} Assume the hypotheses of Proposition \ref{p:properties}, let $x \in \H$ and define the nonempty compact convex set $M(x) := \argmax_{\lambda \in Q} \left\langle \lambda, g(x)\right\rangle$. Then $\mathcal{D}$ has a unique minimizer on $M(x)$, denoted by $\bar{\lambda}(x)  = \argmin_{\lambda \in M(x)} \mathcal{D}(\lambda)$, and $\lambda^{\mu}(x) \to \bar{\lambda}(x)$ as $\mu\to 0$.
	\end{theorem}
	\begin{proof} Fix $\eta \in M(x)$. Notice that optimality conditions for $\lambda^{\mu}(x)$ give 
		\[\left\langle \lambda^{\mu}(x), g(x) \right\rangle - \mu \mathcal{D}\left( \lambda^{\mu}(x) \right) \geq \varphi(x) - \mu\mathcal{D}(\eta).\]
		Since $\left\langle \lambda^{\mu}(x), g(x) \right\rangle \leq \varphi(x)$, 
		\begin{equation}\label{eq:ineq_dual_selection}
			0 \leq \varphi(x) - \left\langle \lambda^{\mu}(x), g(x) \right\rangle \leq \mu \left( \mathcal{D}(\eta) - \mathcal{D}(\lambda^{\mu}(x))\right) \Longrightarrow \mathcal{D}(\lambda^{\mu}(x)) \leq \mathcal{D}(\eta).
		\end{equation} 
		Compactness of $Q$ implies that every sequence $\mu_n \to 0$ has a subsequence such that $\lambda^{\mu_n}(x)$ converges to some $\lambda^0 \in Q$. Notice that \eqref{eq:ineq_dual_selection} implies that $\lambda^0 \in M(x)$, and $\mathcal{D}(\lambda^0) \leq \mathcal{D}(\eta)$ for every $\eta \in M(x)$ thanks to the lower semicontinuity of $\mathcal{D}$. Thus, $\lambda^0$ minimizes $\mathcal{D}$ on $M(x)$. Strong convexity makes the minimizer unique, so every convergent subsequence has the same limit, proving the result. 
	\end{proof}

		For the forthcoming analysis, we collect the assumptions used so far into the following standing hypothesis.
		\begin{hypothesis}\label{hip:general} The set $Q \subset \R_+^m$ is nonempty, compact and convex, $\mathcal{D}$ is lower semicontinuous, nonnegative, and $\sigma$-strongly convex, and $\mathcal{C}  = \sup_Q \mathcal{D}< \infty$. Also, each $g_i$ is convex and Fréchet differentiable and $\nabla g_i$ is Lipschitz continuous on bounded sets of $\mathcal{H}$.  
		\end{hypothesis}

		\section{Inertial dynamics with vanishing damping}\label{sec:dynamics}
		In this section we analyze the well-posedness and the asymptotic behavior of  the nonautonomous inertial system
		\begin{equation}\label{eq:inertial}
			\ddot{x}(t) + \dfrac{\alpha}{t}\dot{x}(t) + \nabla \varphi_{\mu(t)}(x(t))=0,
		\end{equation}
		for $t\geq t_0 >0$ and $\alpha > 0$.  The system is a vanishing-damping inertial gradient flow driven by the time-dependent smooth approximation $\varphi_{\mu(t)}$ of the original nonsmooth objective $\varphi$. The dependence on $t$ through $\mu(t)$ is essential for asymptotic consistency with \eqref{eq:prob_max}, but it also raises specific analytical issues: the Lipschitz modulus of $\nabla\varphi_{\mu}$ typically deteriorates as $1/\mu$, and classical Lyapunov arguments must account for additional terms involving $\dot{\mu}(t)$. We address these points by first proving global well-posedness of the system, which holds true under weaker assumptions than the asymptotic analysis, and then deriving a Lyapunov-type estimate that yields accelerated decay rates and convergence of trajectories under suitable assumptions on the regularization schedule. 
		
		\subsection{Well-posedness of the Inertial System}
		The existence of solutions to the above system, in the unpenalized case and for a function $\varphi$ with Lipschitz-continuous gradient, was established in the seminal work \cite{su2016differential}. In our setting, the argument is slightly different because the mapping $x \mapsto \nabla \varphi_{\mu(t)}(x)$ fails to be globally Lipschitz in general: for each fixed $t$ it is only locally Lipschitz, and it is locally bounded (see Proposition~\ref{p:Lipschitz}). Hence, given an arbitrary compact interval $[t_0,T]\subset \mathbb{R}_+$ we will first obtain a unique maximal solution, in the sense that it is defined on a maximal subinterval $[t_0,\tau_{\max})$ with $\tau_{\max}\le T$ and cannot be extended as a solution within $[t_0,T]$ beyond $\tau_{\max}$. In a second step, we will show that such solutions admit an extension up to $T$, and since $T>t_0$ is arbitrary, this yields a solution defined for all $t\ge t_0$. The proof is given in the appendix. 
		\begin{theorem}\label{teo:global_existence}
			Assume that \ref{hip:general} holds, $\alpha > 0$, and that $\mu \colon [t_0,+\infty)\to(0,+\infty)$ is continuously differentiable.  Then, for  any initial conditions $x(t_0)=x_0\in\mathcal H$ and $\dot x(t_0)=y_0\in\mathcal H$, the problem \eqref{eq:inertial} admits a unique global solution $x\in C^{2}\big([t_0,\infty);\mathcal H\big)$.
		\end{theorem}
		\begin{proof}
			The proof is given in Appendix \ref{thm3.1}.
		\end{proof}
		Notice that the existence result is more general than the following asymptotic analysis: holds for every $\alpha> 0$, the existence of minimizers or the monotonicity of $\mu(t)$. 
		\subsection{Asymptotic Convergence}
		Next, we establish convergence properties of the solutions of \eqref{eq:inertial} toward a minimizer of the original problem \eqref{eq:prob_max}. Although \eqref{eq:inertial} is closely related to accelerated inertial gradient flows studied in, e.g., \cite{su2016differential,attouch2018fast,attouch2019convergence,Attouch2019book,attouch2014dynamical,may2017asymptotic,shi2022understanding}, the presence of a time-dependent vanishing regularization introduces additional difficulties. In particular, if $\argmin\varphi_{\mu} \neq \emptyset$, for a fixed $\mu>0$, trajectories converge to a minimizer of $\varphi_\mu$ when $\alpha >3$, which generally does not coincide with a minimizer of $\varphi$ (see Remark~\ref{r:minimums}). Therefore, one must let $\mu=\mu(t)\downarrow 0$ along time, and the decay of $\mu(t)$ must be controlled.
		
		Throughout this section, assume that $\argmin \varphi \neq \emptyset$. Fix $x^{\ast} \in \argmin \varphi$ (a solution of \eqref{eq:prob_max}). Consider the energy functional 
		$$
		\mathcal{E}(t):=\frac{1}{2}\Vert v(t)\Vert^2+\frac{t^2}{(\alpha-1)^2}Z(t) ,
		$$
		where 
		\begin{equation*}
			v(t):=x(t)-x^{\ast}+\frac{t}{\alpha-1}\dot{x}(t)\quad \textrm{ and } \quad Z(t):=\varphi_{\mu(t)}(x(t))-\inf\varphi + \mathcal{C}\mu(t). 
		\end{equation*}
		As mentioned before, $x^{\ast}$ is not necessarily a minimizer of $\varphi_{\mu(t)}$ (see Remark~\ref{r:minimums}), but we have
		\[Z(t) \geq \varphi_{\mu(t)}(x(t))-\varphi(x(t)) + \mathcal{C}\mu(t) \geq 0,\]
		where the positiveness comes from \eqref{eq:prop1}. Hence, the functional $\mathcal{E}$ is nonnegative and its derivative is controlled by the value gap and the variation of $\mu$. The following estimate is the starting point for the convergence rates.

		\begin{proposition}
			Assume that \ref{hip:general} holds and $\argmin \varphi \neq \emptyset$. Let $\mu\colon [t_0,+\infty)\to(0,+\infty)$ be nonincreasing and continuously differentiable, and let $x(\cdot)$ be the solution of \eqref{eq:inertial} with $\alpha\ge 3$. Then, for all $t\ge t_0$,
			
			\begin{equation}\label{eq:energy_deriv}
				\dot{\mathcal{E}}(t) \leq \dfrac{(3-\alpha)}{(\alpha-1)^2}t Z(t) + \frac{\mathcal{C}}{\alpha-1}t\mu(t).  
			\end{equation}
		\end{proposition}
		\begin{proof} First, by using \eqref{eq:inertial}, we obtain that 
			\begin{equation*}
				\begin{aligned}
					\langle v(t),\dot{v}(t)\rangle =-\frac{t}{\alpha-1}\langle v(t),\nabla \varphi_{\mu(t)}(x(t))\rangle.      
				\end{aligned}
			\end{equation*}
			On the other hand, using \eqref{eq:prop3} and the chain rule, we have
			\begin{equation}\label{eq:deriv_varphi_mu}
				\frac{d}{dt}\varphi_{\mu(t)}(x(t))=\langle\nabla\varphi_{\mu(t)}(x(t)),\dot x(t)\rangle+\dot\mu(t)\partial_\mu\varphi_{\mu(t)}(x(t)) = -\dot{\mu}(t)\mathcal{D}(\lambda^{\mu(t)}(x(t))) +\langle \nabla \varphi_{\mu(t)}(x(t)),\dot{x}(t)\rangle     
			\end{equation}
			
			Therefore, 
			\begin{equation*}
				\begin{aligned}
					\dot{\mathcal{E}}(t)&=\langle v(t),\dot{v}(t)\rangle+\frac{t^2}{(\alpha-1)^2}\dot{Z}(t) + \frac{2t}{(\alpha-1)^2}Z(t)\\
					&=-\frac{t}{\alpha-1}\langle x(t)-x^{\ast},\nabla \varphi_{\mu(t)}(x(t))\rangle +\frac{2t}{(\alpha-1)^2}Z(t)+\frac{t^2}{(\alpha-1)^2}\dot{\mu}(t)\left( \mathcal{C} - \mathcal{D}(\lambda^{\mu(t)}(x(t))) \right)
				\end{aligned}
			\end{equation*}
			Convexity of $x\mapsto \varphi_{\mu(t)}(x)$ yields 
			$$
			\langle x^{\ast} - x(t),\nabla \varphi_{\mu(t)}(x(t))\rangle \leq \varphi_{\mu(t)}(x^{\ast}) -  \varphi_{\mu(t)}(x(t)) \leq \varphi(x^{\ast}) -  \varphi_{\mu(t)}(x(t)) = \mathcal{C}\mu(t) -Z(t) .
			$$
			Using the previous we obtain 
			\[\dot{\mathcal{E}}(t) \leq \dfrac{(3-\alpha)}{(\alpha-1)^2}t Z(t) + \frac{\mathcal{C}}{\alpha-1}t\mu(t) + \frac{t^2}{(\alpha-1)^2}\dot{\mu}(t)\left( \mathcal{C} - \mathcal{D}(\lambda^{\mu(t)}(x(t))) \right).\]
			The last term is nonpositive since $\mathcal{C} = \sup_{\lambda \in Q} \mathcal{D}(\lambda)$ and $\mu$ is nonincreasing, from which we obtain \eqref{eq:energy_deriv}.

		\end{proof}
		The estimate \eqref{eq:energy_deriv} does not imply that $\mathcal{E}$ is nonincreasing, but it still allows one to derive accelerated decay for the regularized function values and, for $\alpha>3$, additional integrability and convergence properties. To guarantee fast convergence rates, we need to impose conditions on the decay of $\mu(t)$. 
		\begin{hypothesis}\label{as:mu_cont} The function $\mu \colon [t_0,+\infty) \to (0,+\infty)$ is continuously differentiable, nonincreasing with $t\mu(t) \in L^1(t_0,\infty)$.
		\end{hypothesis}
		\begin{Lemma}\label{l:mu_decay} Under Hypothesis \ref{as:mu_cont}, $t^2 \mu(t) \to 0$ as $t\to +\infty$.  
		\end{Lemma}
		\begin{proof} Let $\kappa \in (0,1)$ and consider $t$  such that $t \geq t_0 / \kappa $. Since $\mu$ is nonincreasing, we have
			\[\int_{\kappa t}^{t} s\mu(s)\, ds \geq \mu(t)\int_{\kappa t}^{t} s\, ds = \mu(t)\frac{t^2}{2}(1-\kappa^2).\]
			Since $t\mu(t) \in L^1(t_0,\infty)$, the left hand term tends to zero, giving $t^2\mu(t) \to 0$.
		\end{proof}
		\begin{remark} About the assumptions over $\mu(t)$, we can state the following: 
			\begin{itemize} 
				\item  The schedule $\mu(t) = t^{-(2+\delta)}$ with $\delta>0$ satisfies Hypothesis~\ref{as:mu_cont}.
				\item The conditions imposed on the decay of $\mu(t)$ have a clear dynamical interpretation: the smoothing must vanish fast enough that the approximation error $\mathcal{C}\mu(t)$ decays 
				faster than the accelerated rate $t^{-2}$, and the cumulative energy perturbation it induces, 
				of order $\int_{t_0}^{\infty} t\mu(t)\,dt$, must be finite. Monotonicity of $\mu$ ensures that 
				the variation of the smoothing parameter acts dissipatively on the energy.
			\end{itemize}
		\end{remark}
		The integrable energy perturbation yields a rate for both the regularized and the original objective.
		\begin{theorem}\label{t:rates_inertial} Assume that \ref{hip:general} and \ref{as:mu_cont} hold, $\argmin \varphi \neq \emptyset$ and let $x$ solve \eqref{eq:inertial} with $\alpha\ge3$. Then, for every $x^* \in \argmin \varphi$,  $v$ is bounded, and $   Z(t)=\mathcal{O}(t^{-2})$. Consequently,
			\begin{equation}\label{eq:smoothed_O_t2}
				\left|\varphi_{\mu(t)}(x(t))-\inf\varphi\right|=\mathcal{O}(t^{-2})
			\end{equation}
			and
			\begin{equation}\label{eq:original_O_t2}
				0\le\varphi(x(t))-\inf\varphi=\mathcal{O}(t^{-2}).
			\end{equation}
			If $\alpha>3$, then additionally
			\begin{equation}\label{eq:rho_L1}
				\int_{t_0}^{\infty}tZ(t)\,dt<\infty
			\end{equation}
			and
			\begin{equation}\label{eq:values_L1}
				\int_{t_0}^{\infty}
				t\left|\varphi_{\mu(t)}(x(t))-\inf\varphi\right|\,dt<\infty.
			\end{equation}
		\end{theorem}
		\begin{proof} Since $\alpha\geq 3$, we can integrate \eqref{eq:energy_deriv} over $[t_0,t]$ and obtain
			\[\mathcal{E}(t) \leq \mathcal{E}(t_0) + \frac{\mathcal{C}}{(\alpha-1)}\int_{t_0}^{t} s\mu(s)\, ds.\]
			The first bound is obtained combining the previous with $0 \leq \frac{t^2}{(\alpha-1)^2}Z(t) \leq \mathcal{E}(t)$ and Hypothesis \ref{as:mu_cont}. The same argument yields the boundedness of $v(t)$. Moreover, we have 
			\[\vert \varphi_{\mu(t)}(x(t)) - \inf \varphi \vert \leq Z(t) + \mathcal{C}\mu(t),\]
			and since Lemma \ref{l:mu_decay} gives in particular, that $\mu(t) = \mathcal{O}(t^{-2})$, we obtain \eqref{eq:smoothed_O_t2}. Notice that 
			\begin{equation*}
				0 \leq \varphi(x(t)) - \inf \varphi \leq Z(t),
			\end{equation*}
			which gives \eqref{eq:original_O_t2}. If $\alpha > 3$, \eqref{eq:energy_deriv} yields
			\[\dfrac{\alpha-3}{(\alpha-1)^2} \int_{t_0}^t s Z(s)\, ds \leq \mathcal{E}(t_0) - \mathcal{E}(t) + \frac{\mathcal{C}}{\alpha-1}\int_{t_0}^t s\mu(s)\,ds \leq \mathcal{E}(t_0)+ \frac{\mathcal{C}}{\alpha-1}\int_{t_0}^t s\mu(s)\,ds.\]
			Then, \eqref{eq:rho_L1} is obtained thanks to Hypothesis \ref{as:mu_cont}, and \eqref{eq:values_L1} is obtained from $\vert \varphi_{\mu(t)}(x(t)) - \inf\varphi\vert \leq Z + \mathcal{C}\mu$ and \ref{as:mu_cont}. 
		\end{proof}
		
		We now proceed to study further properties of the trajectories and their convergence. For that purpose, let us lighten the notation by defining
		\[\zeta(t):= \varphi_{\mu(t)}(x(t)) - \inf \varphi.\]
		The value estimate also bounds the velocity and, for $\alpha >3$, gives a limit for each distance to a minimizer.

			\begin{proposition}\label{p:estimations_traj} 
				Assume that \ref{hip:general} and \ref{as:mu_cont} hold, and $\argmin \varphi \neq \emptyset$. Let $x(\cdot)$ be a solution of \eqref{eq:inertial} with $\alpha\geq 3$. Then, 
				\[
				\Vert \dot{x}(t) \Vert = \mathcal{O}\!\left(\frac{1}{t}\right).
				\]
				Moreover, if $\alpha>3$, then 
				\begin{equation}\label{eq:int_t_normdotx}
					\int_{t_0}^{+\infty} s\Vert \dot{x}(s)\Vert^2\,ds < +\infty,  
				\end{equation}
				and, for every $x^{\ast}\in \operatorname{argmin}\varphi$, the limit $\lim_{t\to \infty}\Vert x(t) - x^*\Vert$ exists.
			\end{proposition}
			
			\begin{proof}
				Fix $x^*\in\argmin\varphi$ and recall
				\[
				v(t)=x(t)-x^*+\frac{t}{\alpha-1}\dot x(t).
				\]
				By Theorem~\ref{t:rates_inertial}, $v(\cdot)$ is bounded; set $M:=\sup_{t\ge t_0}\|v(t)\|<\infty$.
				Then
				\[
				\dot x(t)=\frac{\alpha-1}{t}\Big(v(t)-(x(t)-x^*)\Big),
				\]
				and, writing $\omega(t):=\|x(t)-x^*\|$, we have
				\[
				\frac{d}{dt}\Big(\frac12\omega^2(t)\Big)=\langle x(t)-x^*,\dot x(t)\rangle
				=\frac{\alpha-1}{t}\big(\langle x(t)-x^*,v(t)\rangle-\|x(t)-x^*\|^2\big)
				\le \frac{\alpha-1}{t}\big(M\omega(t)-\omega^2(t)\big).
				\]
				This implies $\omega(t)\le \max\{\omega(t_0),M\}$ for all $t\ge t_0$, hence
				\[
				\|\dot x(t)\|\le \frac{\alpha-1}{t}\big(\|v(t)\|+\|x(t)-x^*\|\big)\le \frac{C}{t},
				\]
				which proves $\|\dot x(t)\|=\mathcal O(1/t)$.
				
				Assume now $\alpha>3$. Taking the inner product of \eqref{eq:inertial} with $\dot x(t)$ gives
				\[
				\frac12\frac{d}{dt}\|\dot x(t)\|^2+\frac{\alpha}{t}\|\dot x(t)\|^2+\langle\nabla\varphi_{\mu(t)}(x(t)),\dot x(t)\rangle=0,
				\]
				hence
				\[
				\langle\nabla\varphi_{\mu(t)}(x(t)),\dot x(t)\rangle=-\frac12\frac{d}{dt}\|\dot x(t)\|^2-\frac{\alpha}{t}\|\dot x(t)\|^2.
				\]
				Using \eqref{eq:deriv_varphi_mu} yields
				\begin{align*}
					\frac{d}{dt}\big(t^2\zeta(t)\big)
					&=2t\zeta(t)+t^2\dot\zeta(t)\\
					&=2t\zeta(t)-t^2\dot\mu(t)\mathcal D(\lambda^{\mu(t)}(x(t)))
					-\frac{t^2}{2}\frac{d}{dt}\|\dot x(t)\|^2-\alpha t\|\dot x(t)\|^2\\
					&=2t\zeta(t)-t^2\dot\mu(t)\mathcal D(\lambda^{\mu(t)}(x(t)))
					-\frac12\frac{d}{dt}\big(t^2\|\dot x(t)\|^2\big)-(\alpha-1)t\|\dot x(t)\|^2.
				\end{align*}
				Rearranging, we obtain
				\[
				(\alpha-1)t\|\dot x(t)\|^2
				=2t\zeta(t)-t^2\dot\mu(t)\mathcal D(\lambda^{\mu(t)}(x(t)))
				-\frac12\frac{d}{dt}\big(t^2\|\dot x(t)\|^2\big)-\frac{d}{dt}\big(t^2\zeta(t)\big).
				\]
				Integrating on $[t_0,t]$ gives
				\begin{align*}
					(\alpha-1)\int_{t_0}^t s\|\dot x(s)\|^2ds
					&=2\int_{t_0}^t s\zeta(s)\,ds
					-\int_{t_0}^t s^2\dot\mu(s)\mathcal D(\lambda^{\mu(s)}(x(s)))\,ds\\
					&\quad-\frac12\Big[s^2\|\dot x(s)\|^2\Big]_{t_0}^t-\Big[s^2\zeta(s)\Big]_{t_0}^t.
				\end{align*}
				
				Now, $\dot\mu\le 0$ and $\mathcal D\le \mathcal C$, so  Hypothesis~\ref{as:mu_cont} and Lemma \ref{l:mu_decay} give
				\[
				-\int_{t_0}^\infty t^2\dot\mu(t)\mathcal D(\cdot)\,dt
				\le- \mathcal C\int_{t_0}^\infty t^2\dot\mu(t)\,dt  = \mathcal{C}\left( t_0^2\mu(t_0) - \lim_{t \to \infty} t^2 \mu(t) + 2 \int_{t_0}^{\infty} t\mu(t)\, dt \right)<\infty
				.
				\]
				Moreover, \eqref{eq:values_L1} gives $t|\zeta(t)|\in L^1(t_0,\infty)$, hence $\int_{t_0}^{\infty}  s\zeta(s)\,ds$ is finite.
				Finally, $t^2\|\dot x(t)\|^2$ is bounded because $\|\dot x(t)\|=\mathcal O(1/t)$, and $t^2\zeta(t)$ is bounded by \eqref{eq:smoothed_O_t2}. Letting $t\to\infty$ yields \eqref{eq:int_t_normdotx}.

				We now prove the existence of $\lim_{t\to\infty}\|x(t)-x^*\|$.
				Define
				\[
				u(t):=\frac12\|x(t)-x^*\|^2,\qquad 
				h(t):=\|v(t)\|^2-\frac{t^2}{(\alpha-1)^2}\|\dot x(t)\|^2.
				\]
				A direct expansion of $\|v(t)\|^2$ gives the identity
				\[
				h(t)=\|x(t)-x^*\|^2+\frac{2t}{\alpha-1}\langle x(t)-x^*,\dot x(t)\rangle
				=2u(t)+\frac{2t}{\alpha-1}\dot u(t),
				\]
				hence
				\begin{equation}\label{eq:u_h_link}
					\dot u(t)+\frac{\alpha-1}{t}u(t)=\frac{\alpha-1}{2t}\,h(t).
				\end{equation}
				Next, differentiating $h$  yields
				\[
				\dot h(t)=\frac{2t}{\alpha-1}\Big(\|\dot x(t)\|^2-\langle x(t)-x^*,\nabla\varphi_{\mu(t)}(x(t))\rangle\Big).
				\]
				By convexity of $\varphi_{\mu(t)}$,
				\[
				\langle \nabla\varphi_{\mu(t)}(x(t)),x(t)-x^*\rangle \ge \varphi_{\mu(t)}(x(t))-\varphi_{\mu(t)}(x^*).
				\]
				Since $\varphi_{\mu(t)}(x^*)\le \varphi(x^*)=\inf\varphi$, we get
				\[
				-\langle \nabla\varphi_{\mu(t)}(x(t)), x(t)-x^*\rangle \le -(\varphi_{\mu(t)}(x(t))-\inf\varphi)= -\zeta(t),
				\]
				and therefore,
				\[
				\max\lbrace\dot h(t),0 \rbrace
				\le \frac{2t}{\alpha-1}\Big(\|\dot x(t)\|^2+|\zeta(t)|\Big).
				\]
				The right-hand side is integrable on $(t_0,\infty)$ thanks to \eqref{eq:int_t_normdotx} and since $t|\zeta(t)|\in L^1(t_0,+\infty)$.
				Since $h(\cdot)$ is bounded (because $v$ is bounded and $t\dot x$ is bounded), one gets $\dot h\in L^1$. The previous implies that
				$\lim_{t\to\infty}h(t)=:h_\infty$ exists.
				
				Finally, solve \eqref{eq:u_h_link}: multiplying by $t^{\alpha-1}$ gives 
				$$u(t)=\left(\frac{t_0}{t}\right)^{\alpha-1}u(t_0)
				+\frac{\alpha-1}{2t^{\alpha-1}}
				\int_{t_0}^{t}s^{\alpha-2}h(s)\,ds.$$
				Since $\alpha>3$ and $h(t)\to h_\infty$, the weighted average converges to $h_\infty/2$. Thus $u(t)\to h_\infty/2$, and $\|x(t)-x^*\|$ has a limit.

			\end{proof}
			
			The shifted mechanical energy is nonnegative and nonincreasing. Its weighted integrability gives the sharper decay needed for weak convergence.
			\begin{theorem}\label{t:traj} Assume that \ref{hip:general} and \ref{as:mu_cont} hold, and $\argmin \varphi \neq \emptyset$. Let $x(\cdot)$ be a solution of \eqref{eq:inertial} with $\alpha> 3$. Then the following assertions hold:
				\begin{enumerate}[label={\rm (\roman*)}]
					\item The regularized value gap satisfies 
					$$\varphi_{\mu(t)}(x(t))-\inf \varphi=o\left(\frac{1}{t^2}\right).
					$$
					\item The trajectory $x(t)$ converges weakly to some $x_{\infty} \in \argmin \varphi$. 
					\item The original value gap satisfies 
					$$\varphi(x(t))-\inf \varphi
					= o\!\left(\frac{1}{t^{2}}\right).$$
					\item The velocity satisfies $\Vert \dot{x}(t)\Vert =o\!\left(1/t\right)$ as $t\to \infty$.
				\end{enumerate}
			\end{theorem}
			
			\begin{proof} Set
				\[\mathcal{W}(t) = \dfrac{1}{2}\Vert\dot{x}(t) \Vert^2 + \zeta(t) + \mathcal{C}\mu(t). \]
				By \eqref{eq:prop1}, $\mathcal{W}(t) \geq 0$.  The dynamics \eqref{eq:inertial} and the chain rule \eqref{eq:deriv_varphi_mu} give
				$$
				\dot{\mathcal W}(t)
				=-\frac{\alpha}{t}\Vert\dot x(t)\Vert^2
				+\dot\mu(t)\bigl(\mathcal C-
				\mathcal D(\lambda^{\mu(t)}(x(t)))\bigr)\leq0,
				$$
				since $\dot\mu\leq0$ and $0\leq\mathcal D\leq\mathcal C$.
				Moreover, \eqref{eq:values_L1}, \eqref{eq:int_t_normdotx}, and Hypothesis~\ref{as:mu_cont} imply
				$$
				\int_{t_0}^{+\infty}t\mathcal W(t)\,dt
				\leq\frac12\int_{t_0}^{+\infty}t\Vert\dot x(t)\Vert^2\,dt
				+\int_{t_0}^{+\infty}t|\zeta(t)|\,dt
				+\mathcal C\int_{t_0}^{+\infty}t\mu(t)\,dt<+\infty.
				$$
				Let $\kappa \in (0,1)$ and $t \geq t_0/\kappa$.  Monotonicity of $\mathcal W$ yields
				$$
				\frac{t^2}{2}\mathcal{W}(t)(1-\kappa^2)
				\leq\int_{\kappa t}^{t}s\mathcal W(s)\,ds\longrightarrow0.
				$$
				Hence $t^2\mathcal W(t)\to0$. The definitions and \eqref{eq:prop1} give
				$$
				-\mathcal C\mu(t)\leq\zeta(t)\leq\mathcal W(t),
				\qquad
				0\leq\varphi(x(t))-\inf\varphi\leq\mathcal W(t),
				\qquad
				\frac12\Vert\dot x(t)\Vert^2\leq\mathcal W(t).
				$$
				Since $t^2\mu(t)\to0$ by Lemma~\ref{l:mu_decay}, these inequalities prove (i), (iii), and (iv).
				For (ii), Proposition~\ref{p:estimations_traj} shows that $\Vert x(t)-x^*\Vert$ has a limit for every $x^*\in\argmin\varphi$. In particular, the trajectory is bounded. Let $t_n\to+\infty$ and suppose $x(t_n)\rightharpoonup\bar x$. The function $\varphi$ is convex and continuous under \ref{hip:general}, hence weakly lower semicontinuous. By (iii),
				$$
				\varphi(\bar x)
				\leq\liminf_{n\to\infty}\varphi(x(t_n))
				=\inf\varphi.
				$$
				Thus every weak sequential cluster point belongs to $\argmin\varphi$. Opial's lemma, recalled as Lemma~\ref{l:opial}, gives $x(t)\rightharpoonup x_\infty\in\argmin\varphi$.
				
			\end{proof}

			\section{Applications}\label{eq:appl}
			In this section, we present illustrative examples of the penalty-based smoothing framework and the associated inertial dynamics with vanishing regularization. We first show how the construction applies to multiobjective optimization through Chebyshev (max) scalarization, leading to convergence toward a minimizer of the chosen scalarization, which is weakly Pareto optimal. We then discuss the distributionally robust optimization (DRO) setting, where KL-based penalties produce tractable smooth surrogates for worst-case expectations over ambiguity sets. All numerical experiments integrate \eqref{eq:inertial}, written as a first-order system in $(x,\dot x)$, with \texttt{solve\_ivp} from \texttt{scipy.integrate}, using the explicit Runge--Kutta method \texttt{RK45} with \texttt{rtol=1e-8} and \texttt{atol=1e-10}. Exponential weights are evaluated in log-sum-exp form, after subtracting the largest exponent. Whenever randomness is involved, reproducibility is ensured by fixing the random number generator seed via \texttt{numpy.random.default\_rng(2)}.

			\subsection{Multiobjective Optimization}
			Consider the multiobjective optimization problem
			\begin{equation}\label{eq:prob}
				\min_{x \in \mathcal{H}} G(x):= \paren{g_1(x),g_2(x),\ldots,g_m(x)},
			\end{equation}
			where each $g_i\colon \mathcal{H} \to \R$ is convex and Fr\'echet differentiable. We refer to \cite{MR1784937,ehrgott2005multicriteria} for standard background on multiobjective optimization. A point $x^*$ is called \textit{Pareto optimal} if there is no $x\in \mathcal{H}$ such that
			\[g_i(x) \leq g_i(x^*) \quad \text{ for all } \quad i=1,\ldots,m, \quad \text{ and } \quad g_j(x) < g_j(x^*),\quad \text{for some }  j.\]
			Informally, this means that one cannot improve one objective without worsening at least one other objective. The set of Pareto optimal points is denoted by $P$ (Pareto set). A point $x^*$ is called \textit{weakly Pareto optimal} if there is no $x \in \mathcal{H}$ such that $g_i(x) < g_i(x^*)$ for all $i = 1,\ldots,m$. The set of weakly Pareto optimal points is denoted by $P_w$ (weak Pareto set). Clearly, $P \subset P_w$.

			A point $x^* \in \mathcal H$ is called \textit{Pareto critical} (or \emph{weakly stationary}) if there exists $\lambda \in \Delta_m$ such that
			\begin{equation}\label{eq:cond_opt}
				\sum_{i=1}^m \lambda_i \nabla g_i(x^*) = 0 .
			\end{equation}
			In the unconstrained convex differentiable setting \eqref{eq:prob}, Pareto criticality characterizes weak Pareto optimality. More precisely, $x^*$ is weakly Pareto optimal if and only if
			\begin{equation}\label{eq:weakpareto}
				0 \in \operatorname{conv}\{\nabla g_1(x^*),\ldots,\nabla g_m(x^*)\},
			\end{equation}
			or, equivalently, if \eqref{eq:cond_opt} holds for some $\lambda \in \Delta_m$  (see Ehrgott~\cite[Theorem~3.21]{ehrgott2005multicriteria}). In the terminology of Ehrgott, Pareto critical points correspond to \emph{weakly efficient} solutions.

			We connect \eqref{eq:prob} to our supremum framework via the classical Chebyshev (max) scalarization
			\[\min \varphi(x):= \max_{i=1,\ldots,m}g_i(x).
			\]
			Any minimizer of $\varphi$ is weakly Pareto optimal: indeed, if $x^{\ast}\in \operatorname{argmin}\varphi$ and there existed $x$ with $g_i(x)<g_i(x^{\ast})$ for all $i$, then $\max_i g_i(x)<\max_i g_i(x^{\ast})$, contradicting optimality of $x^{\ast}$. Moreover, this scalarization fits \eqref{eq:prob_max} with $Q=\Delta_m$, since $\varphi(x)=\max_{\lambda\in \Delta_m} \sum_{i=1}^m \lambda_i g_i(x)$. If $g_i$ are differentiable, then the subdifferential of $\varphi$ is
			\begin{equation*}
				\partial \varphi(x)=\operatorname{conv}\left\{\nabla g_i(x) \big|\ g_i(x)=\varphi(x)\right\},
			\end{equation*}
			so the first-order condition $0\in \partial \varphi(x^{\star})$ implies \eqref{eq:weakpareto}, hence $x^{\star}\in P_w$. In particular, $\argmin \varphi \subset P_w$. 
			To apply Theorems~\ref{t:rates_inertial} and~\ref{t:traj} to this scalarization, suppose in addition that each $\nabla g_i$ is Lipschitz continuous on bounded subsets of $\mathcal H$, that $\operatorname{argmin}\varphi\neq\emptyset$, and that $\mathcal D\colon\Delta_m\to\R$ is nonnegative, lower semicontinuous, $\sigma$-strongly convex, and bounded above. Then Hypothesis~\ref{hip:general} holds with $Q=\Delta_m$. If $\mu$ satisfies Hypothesis~\ref{as:mu_cont}, Theorem~\ref{t:rates_inertial} yields $\varphi(x(t))-\inf\varphi=\mathcal O(t^{-2})$ for $\alpha\geq3$. For $\alpha>3$, Theorem~\ref{t:traj} yields $o(t^{-2})$ for this gap and weak convergence of $x(t)$ to a point of $\operatorname{argmin}\varphi\subset P_w$. For related accelerated continuous-time models and efficient discretizations designed directly for multiobjective problems (smooth and nonsmooth), we refer to \cite{BotSonntag2026JMAA,SonntagEtAl2024Nonsmooth,sonntag2024fast,sonntagpeitz2024fast}.
			
			To illustrate our method, consider $x \in \R^2$ and the multiobjective optimization problem with three objectives given by
			\[a_1 = (0,0), \quad a_2 = (3,0), \quad a_3 = (1,2), \quad g_i(x) = \frac{1}{2}\Vert x - a_i\Vert^2, \quad i=1,2,3.\]
			Each $g_i$ has a globally $1$-Lipschitz gradient. The remaining conditions of Hypothesis~\ref{hip:general} are verified below. We then consider the problem of minimizing the objective 
			\[\varphi(x) = \max_{i =1,2,3} g_i(x).\]
			To that end, we use the smoothing given by the entropic penalty relative to the uniform distribution, i.e., 
			\[ \mathcal{D}(\lambda) = \sum_{i=1}^m \lambda_i \log(m \lambda_i) \Longrightarrow \varphi_\mu(x) =\mu \log \left(\frac{1}{m} \sum_{i=1}^m e^{g_i(x)/\mu} \right).\]
			With the convention $0\log 0=0$, the penalty $\mathcal D$ is continuous and nonnegative on $\Delta_3$, is $1$-strongly convex with respect to the Euclidean norm, and satisfies $\sup_{\Delta_3}\mathcal D=\log 3$. Moreover, $\varphi$ is $1$-strongly convex and coercive, since each $g_i$ has these properties. Thus $\operatorname{argmin}\varphi$ is a singleton and Hypothesis~\ref{hip:general} holds. The dual weights are given by 
			\[\lambda^{\mu}_i (x) = \dfrac{\exp(g_i(x)/\mu)}{\sum_{j=1}^m \exp(g_j(x)/\mu)}.\]
			Notice that, in this particular example, the Pareto set can be computed as $P = \text{conv}\left\lbrace a_1,a_2,a_3\right\rbrace$. Indeed, if $x \in P$, then $x = \sum_{i=1}^3 \theta_i a_i$, with $\theta \in \Delta_3$. Then, for every $y \in \R^2$, 
			\[\sum_{i=1}^3 \theta_i \left(g_i(y) - g_i(x) \right) = \frac{1}{2}\Vert y-x \Vert^2.\]
			Hence, there is no way of improving the objectives with respect to $x$. Conversely, if $x \notin P$, and taking $p = \operatorname{proj}_P(x)$ as the projection onto $P$, we get 
			\[\Vert p -a_i \Vert^2 < \Vert x -a_i \Vert^2, \quad i=1,2,3, \]
			so the objectives can be improved at $p$, and hence, $x$ is not Pareto. The unique minimizer can be identified explicitly. Noticing that $P$ is an acute triangle, we can state that the solution is attained at the circumcenter, i.e. the solution of the equal-distance equations. A direct computation gives 
			\[x^* = \left( \frac{3}{2}, \frac{1}{2}\right), \quad \inf \varphi = \frac{5}{4}.\]
			For the proposed regularization, we solve system \eqref{eq:inertial} over the interval $[1,25]$, with starting point $x_0 =(3.5, 3.5)$ and initial velocity equal to $(0,0)$. We consider the regularization function $\mu(t) = \mu_0 t^{-r}$, with $\mu_0 =3$, and $r=2.2$. Since $r>2$, this schedule satisfies Hypothesis \ref{as:mu_cont}. Figure \ref{fig:triangle} reports the evolution of the function values error and the trajectories $(x_1(t),x_2(t))$ for several values of $\alpha$.  The observed behavior illustrates the theoretical prediction: trajectories of \eqref{eq:inertial} converge to $x^*$, the minimizer of the unweighted Chebyshev scalarization. Since this minimizer is unique, $x^*$ is Pareto optimal; other Pareto points would require other weights or reference points in the scalarization.
			\begin{figure}[h]
				\centering
				\subfigure[Objective value error for different values of $\alpha$.]{\includegraphics[width=0.48\textwidth]{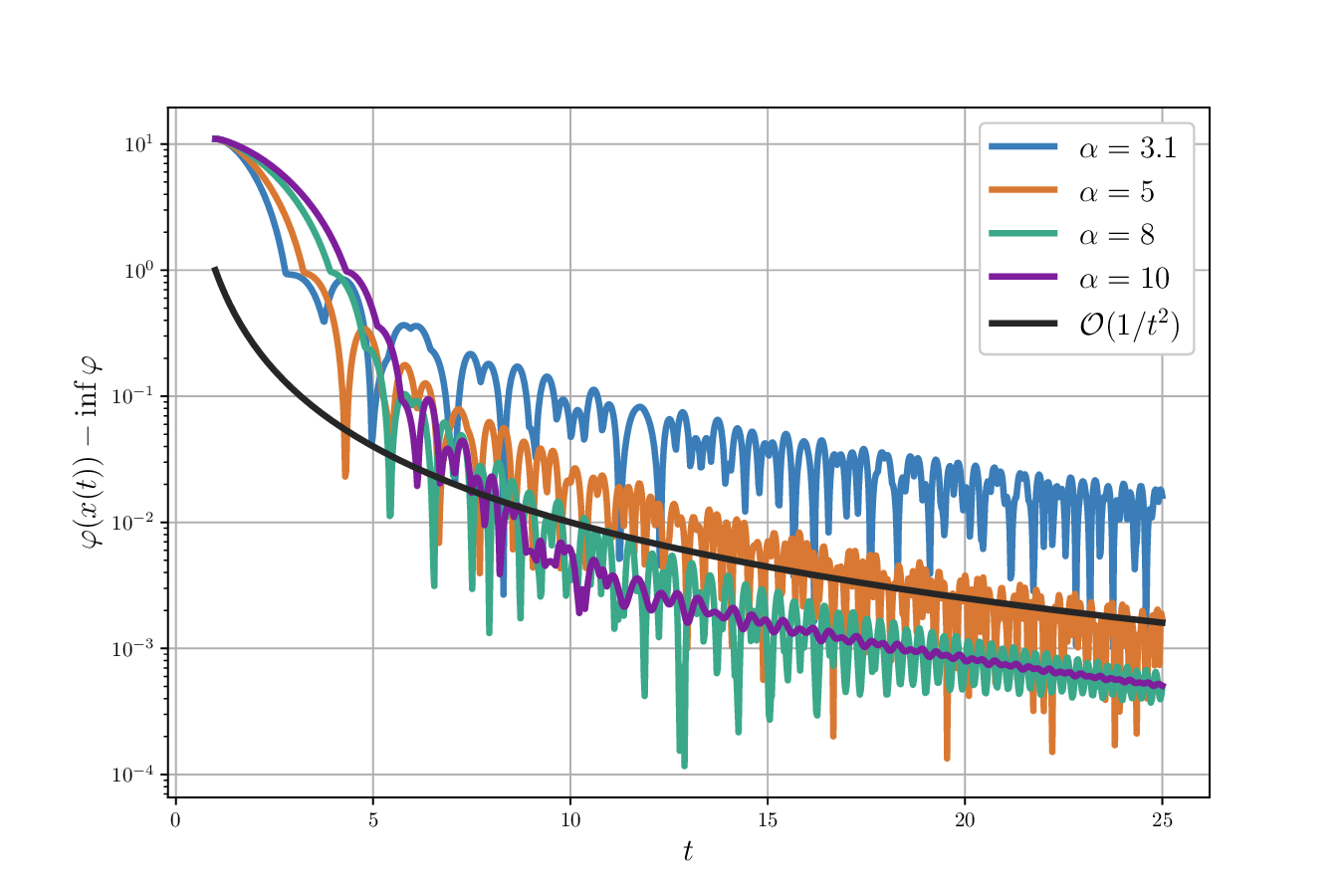} }
				\subfigure[Trajectories for different values of $\alpha$.]{\includegraphics[width=0.48\textwidth]{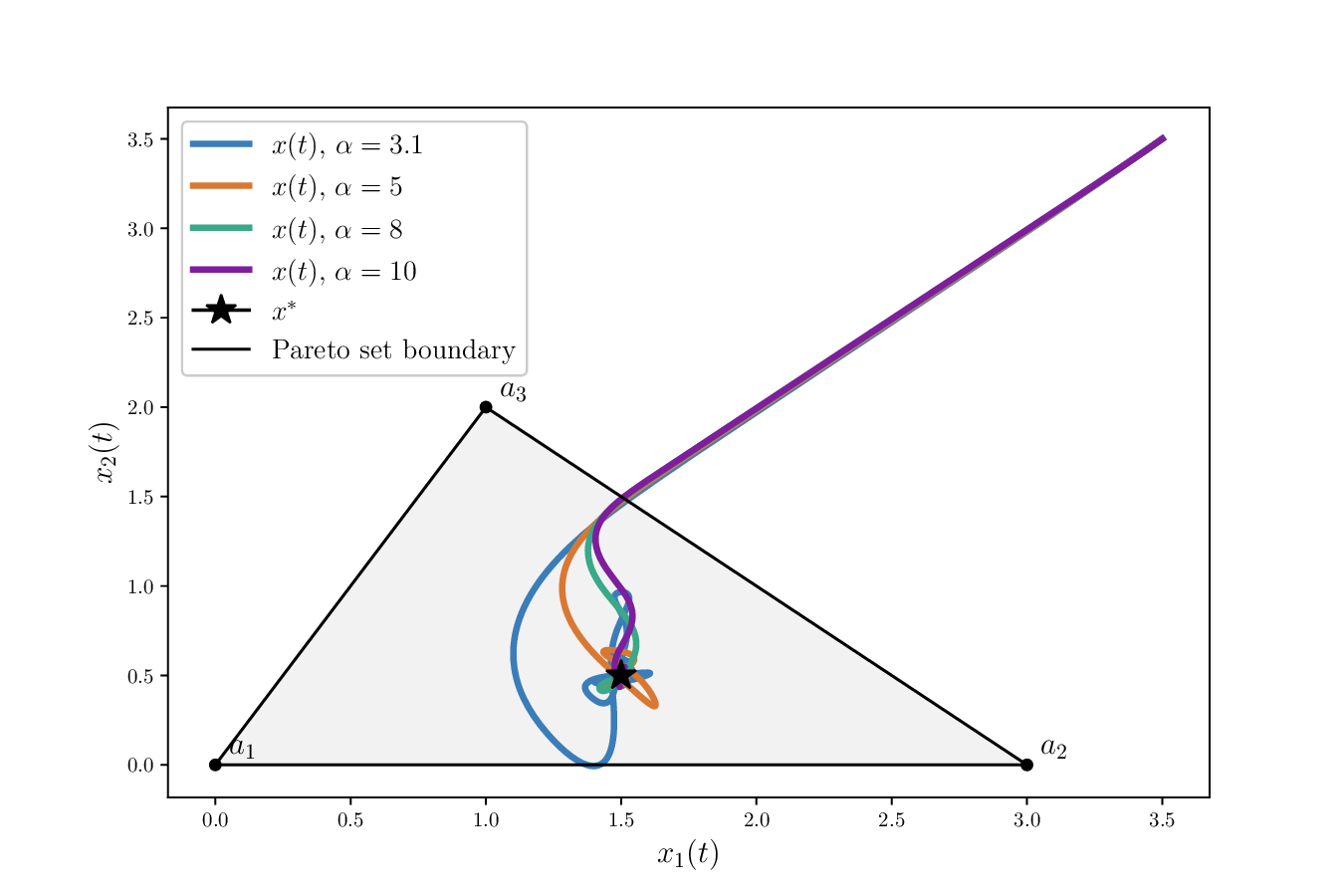}}
				\caption{Multiobjective example with three quadratic objectives.}
				\label{fig:triangle} 
			\end{figure}

			\subsection{Distributionally Robust Optimization}
			We now illustrate how the penalty-based smoothing framework naturally covers  \emph{Distributionally Robust Optimization} (DRO). We refer to \cite{MR4362585,MR4379565,MR3294550} for background. DRO addresses decision-making under uncertainty when the data-generating distribution is unknown or only partially specified. Instead of postulating a single probabilistic model, one assumes that the true distribution belongs to a prescribed \emph{ambiguity set} constructed from the available information, and optimizes against the worst-case distribution within that set. This yields robust decisions that bridge classical stochastic optimization and worst-case robust optimization.
			
			Let $Q=\mathcal P\subset\Delta_m$ be a nonempty closed convex ambiguity set, and let $\xi$ take values in $\{\xi_1,\ldots,\xi_m\}$. Write $g_i(x)=f_i(x):=f(x,\xi_i)$ for the cost of scenario $i$. The function in \eqref{eq:prob_max} is then
			$$
			\varphi(x)=\sup_{p\in\mathcal P}\sum_{i=1}^m p_i f_i(x).
			$$
			Assume that $\mathcal P$ contains a prior $u$ with $u_i>0$ for every $i$. We use the entropic penalty
			$$
			\mathcal D(p)=\operatorname{KL}(p\Vert u)
			:=\sum_{i=1}^m p_i\log\left(\frac{p_i}{u_i}\right),
			\qquad 0\log 0:=0.
			$$
			It is continuous and $1$-strongly convex on $\mathcal P$, and $\mathcal D(u)=0$. Moreover, $0\leq\mathcal D(p)\leq-\log(\min_i u_i)$ for $p\in\mathcal P$. Thus this choice satisfies the penalty assumptions, and
			$$
			\varphi_\mu(x)=\sup_{p\in\mathcal P}
			\left\{\sum_{i=1}^m p_i f_i(x)
			-\mu\sum_{i=1}^m p_i\log\left(\frac{p_i}{u_i}\right)\right\}.
			$$
			If the functions $f_i$ have the convexity and differentiability properties required in \ref{hip:general}, and if $\operatorname{argmin}\varphi\neq\emptyset$, the results for the inertial system apply to this model.
			For an explicit characterization, consider the moment-constrained ambiguity set
			\begin{equation}\label{eq:amb_set}
				\mathcal P=\left\{p\in\mathbb R^m: Ap=b,\ \langle\mathbbm 1,p\rangle=1,\ p\geq0\right\},
			\end{equation}
			where $A\in\mathbb R^{d\times m}$ and $b\in\mathbb R^d$. Such sets describe partial statistical information, such as moments or marginals; see \cite{MR4362585,MR3294550}. The prior must satisfy $Au=b$, $\langle\mathbbm 1,u\rangle=1$, and $u_i>0$ for every $i$. This condition supplies the strictly positive feasible point used below.
			The positive prior also ensures that the regularized maximizer has no zero coordinates.
			\begin{remark}\label{rem:interiority_DRO}
				For every $\mu>0$ and $x$, the regularized maximizer $p^\mu(x)$ is unique and satisfies $p_i^\mu(x)>0$ for all $i$. Indeed, if a coordinate of $p^\mu(x)$ were zero, the feasible segment $(1-\varepsilon)p^\mu(x)+\varepsilon u$ would increase the objective by a positive term of order $-\mu\varepsilon\log\varepsilon$, while the remaining terms change by $O(\varepsilon)$. Hence $p^\mu(x)\in\operatorname{ri}(\mathcal P)$ and the nonnegativity constraints are inactive.
			\end{remark}
			The optimality condition for $p^\mu(x)$ is
			\begin{equation*}
				\frac{1}{\mu}f(x)\in\nabla\mathcal D\bigl(p^\mu(x)\bigr)+N_{\mathcal P}\bigl(p^\mu(x)\bigr),
			\end{equation*}
			where $f(x):=(f_1(x),\ldots,f_m(x))$, $(\nabla\mathcal D(p))_i=1+\log(p_i/u_i)$, and $N_{\mathcal P}$ is the normal cone. By Remark~\ref{rem:interiority_DRO}, only the equality constraints contribute to the normal cone at $p^\mu(x)$. Thus there are multipliers $(\theta^\mu(x),\eta^\mu(x))\in\mathbb R^d\times\mathbb R$ such that
			$$f_i(x)-\mu\left(1+\log\frac{p_i^\mu(x)}{u_i}\right)
			-(A^\top\theta^\mu(x))_i-\eta^\mu(x)=0,
			\qquad i=1,\ldots,m.$$
			Normalization gives the exponential tilt
			\begin{equation}\label{eq:p_DRO}
				p_i^\mu(x)=
				\frac{u_i\exp\left(\frac{f_i(x)-(A^\top\theta^\mu(x))_i}{\mu}\right)}
				{\sum_{j=1}^m u_j\exp\left(\frac{f_j(x)-(A^\top\theta^\mu(x))_j}{\mu}\right)},
				\qquad i=1,\ldots,m,
			\end{equation}
			where $\theta^\mu(x)$ is chosen so that $Ap^\mu(x)=b$. The multipliers need not be unique when the equality constraints are redundant, but $p^\mu(x)$ is unique. The regularized value is
			$$\varphi_\mu(x)=\sum_{i=1}^m p_i^\mu(x)f_i(x)
			-\mu\sum_{i=1}^m p_i^\mu(x)\log\left(\frac{p_i^\mu(x)}{u_i}\right).$$
			
			This places KL-regularized DRO within our smoothing framework and allows us to study inertial dynamics driven by $\nabla\varphi_{\mu(t)}$ with a time-vanishing regularization parameter.
			
			To  numerically illustrate our approach, we consider a DRO instance with $m$ scenarios and a decision variable $x\in\R^n$.  The scenario-dependent cost functions are defined as
			\[f_i(x) = \frac{1}{2}(x-d_i)^\top \mathcal{S}_i (x-d_i) + e_i, \quad i=1,\ldots,m,
			\]
			where $d_i \in \R^n$, $\mathcal{S}_i \in \R^{n\times n}$ is symmetric positive definite, and $e_i \in \R$, for each $i = 1,\ldots,m$. The parameters are generated as follows. For each $i=1,\ldots,m$, we generate a vector $s_i \in \R^{n}$ whose entries are independently and uniformly distributed on $[\tfrac{1}{2}, 2]$, and define $\mathcal{S}_i \in \R^{n\times n}$ as the diagonal matrix with diagonal $s_i$. The vectors $d_i$ are sampled independently from a standard normal distribution in $\R^n$, and the scalars $e_i$ are drawn independently from the uniform distribution on $[-0.2,0.2]$.    We take $n=5$, $m=6$, and consider the ambiguity set \eqref{eq:amb_set}  with  
			\[A= \begin{pmatrix}
				1 & -1 & 0 & 0 & 0 & 0\\
				0 & 1 &-1 & 0 & 0 & 0\\
			\end{pmatrix},\quad b = \begin{pmatrix}
				0 \\ 0
			\end{pmatrix}.\]
			Here the uniform prior $u = (1/6,\ldots,1/6)$ belongs to $\mathcal{P}$ and is strictly positive. Since $\mathcal{S}_i \succeq \frac{1}{2}I$ for every $i$, the function $\varphi$ is strongly convex and has a unique minimizer. Thus, this instance satisfies \ref{hip:general} and is covered by the preceding formulas. We then solve the inertial system \eqref{eq:inertial} on $t\in[1,20]$ with initial conditions 
			$x(1)=0$ and $\dot x(1)=0$, $\mu(t)=t^{-2.2}$, and $\alpha\in\{3.1,5,8,10\}$. At each evaluation, $p^{\mu}(x)$ is obtained from \eqref{eq:p_DRO}, where 
			$\theta^\mu(x)$ minimizes the smooth convex dual function 
			$\theta\mapsto\langle b,\theta\rangle+\mu\log\sum_{i=1}^m u_i\exp\big((f_i(x)-(A^\top\theta)_i)/\mu\big)$, 
			computed by a warm-started Newton method. The Newton iteration  stops when $\|b-Ap\|_\infty<10^{-13}$ or when the Newton step reaches round-off level; in all runs, the final residual was below $2\times10^{-11}$. The reference values $x^*$ and $\inf\varphi$ are computed with SLSQP on the smooth epigraph reformulation $\min_{x,s}\{s:\langle v,f(x)\rangle\le s\}$ over the four vertices $v$ of $\mathcal P$, which also gives $\varphi(x(t))$ exactly as a maximum over these vertices. Their accuracy is several orders of magnitude below the plotted quantities.  Figure~\ref{fig:sols_DR} reports $\varphi(x(t))-\inf\varphi$ and $\|x(t)-x^\ast\|$, in agreement with Theorems~\ref{t:rates_inertial} and~\ref{t:traj}. Since $\varphi$ is $\tfrac{1}{2}$-strongly convex, $\|x(t)-x^*\|^2\le 4(\varphi(x(t))-\inf\varphi)$, so the distance decays at least like $1/t$.

			\begin{figure}[h]
				\centering
				\subfigure[Objective value error for different values of $\alpha$.]{\includegraphics[width=0.48\textwidth]{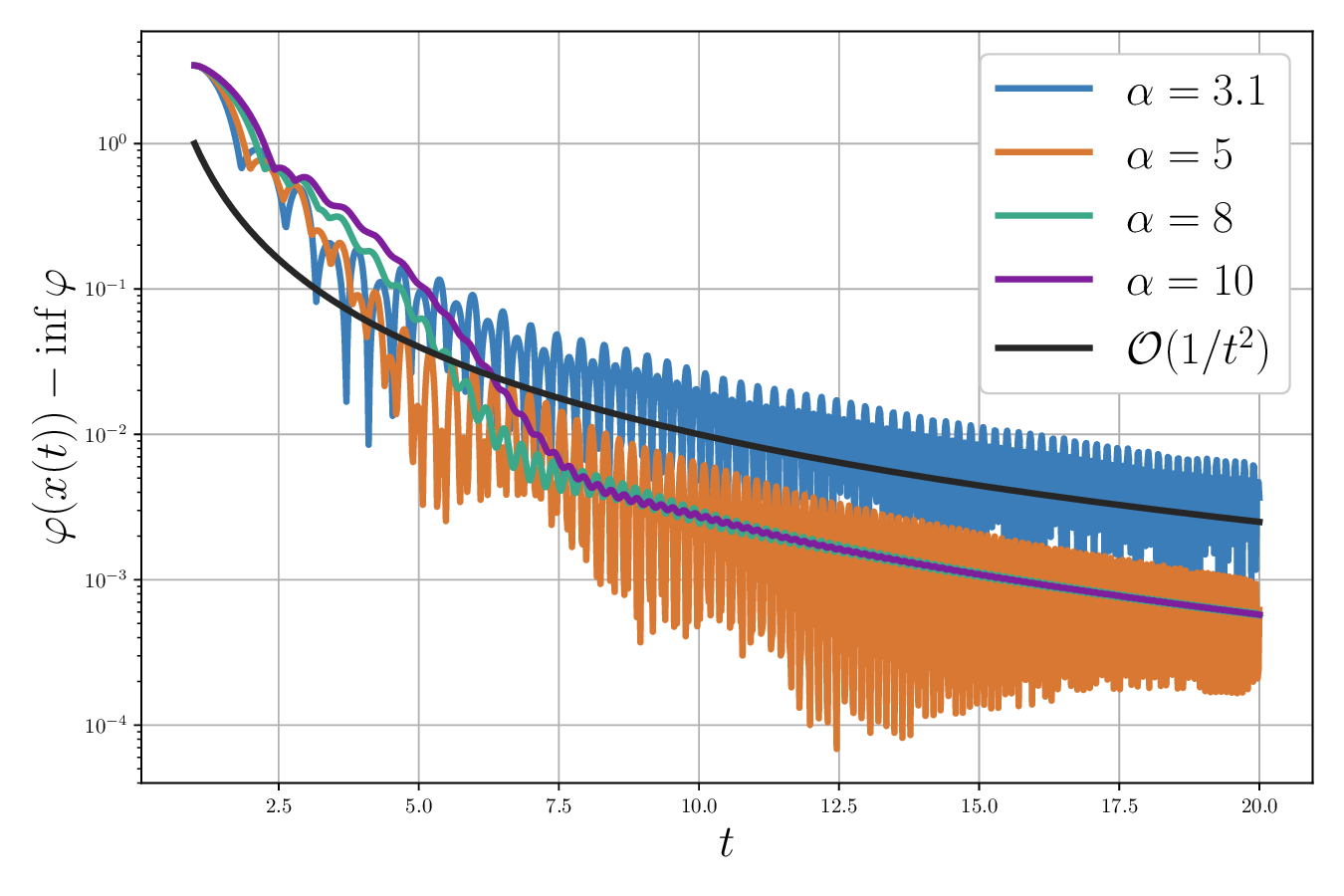} }
				\subfigure[Trajectories error for different values of $\alpha$.]{\includegraphics[width=0.48\textwidth]{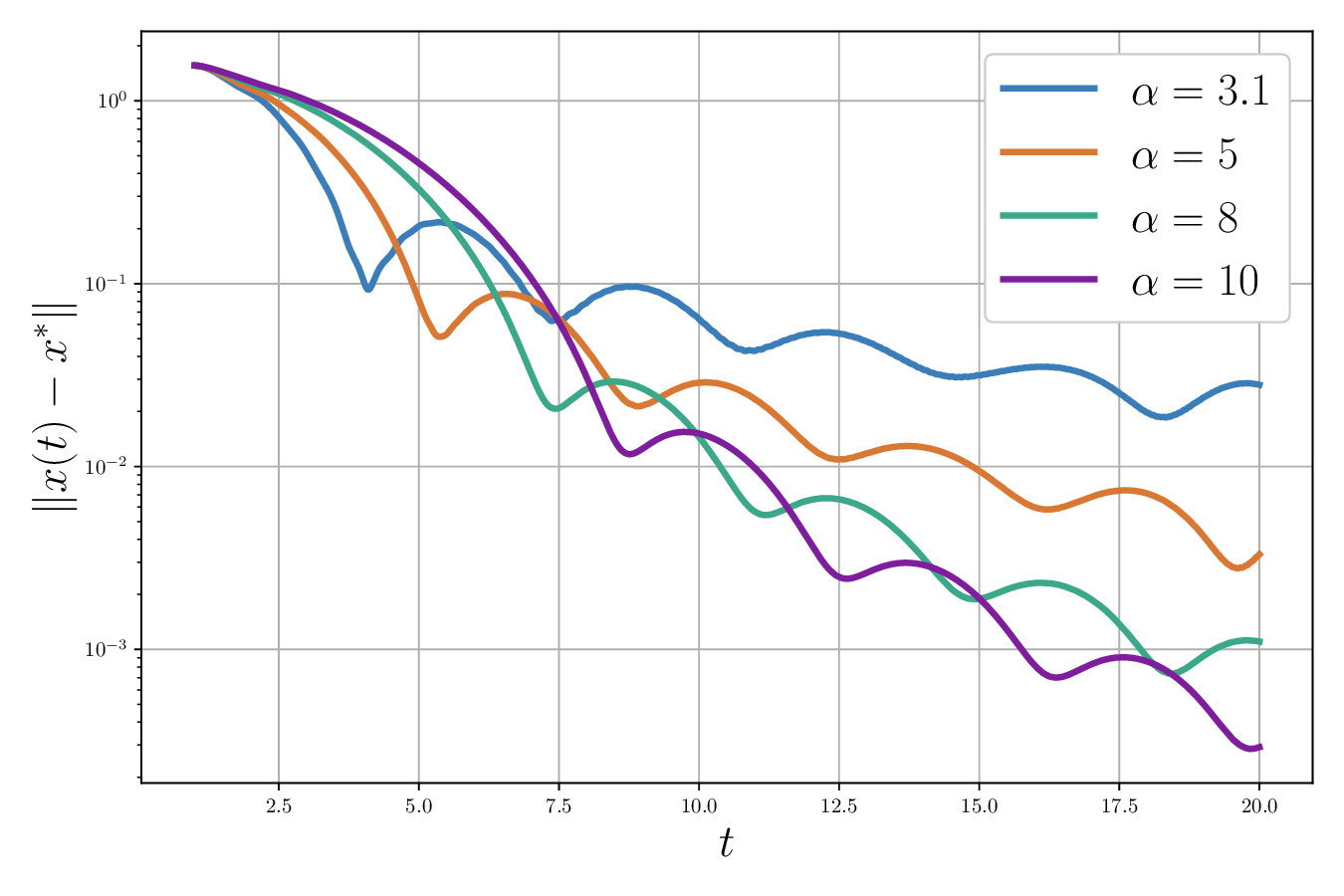}}
				\caption{DRO experiment with quadratic cost functions.}
				\label{fig:sols_DR}
			\end{figure}
			
			\section{Concluding Remarks}
			
			This work develops a penalty-based smoothing framework for convex supremum functions with a finite-dimensional dual representation and investigates accelerated inertial dynamics driven by a \emph{vanishing} regularization parameter. By introducing the regularized family $\varphi_\mu$, the supremum problem is replaced by a smooth convex surrogate whose gradient admits an explicit envelope formula. The resulting approximation is quantitatively controlled from below, with a uniform error of order $\mu$, and the Lipschitz constant of $\nabla\varphi_\mu$ on bounded sets is controlled by a term of order $1/\mu$, making explicit the classical smoothness-accuracy tradeoff. The main dynamical contribution is the analysis of the nonautonomous inertial system
			\[
			\ddot x(t)+\frac{\alpha}{t}\dot x(t)+\nabla \varphi_{\mu(t)}(x(t))=0,
			\]
			with $\mu(t)\downarrow 0$. We prove its global well-posedness for every $\alpha>0$ using only local gradient bounds, a setting not covered by the global Lipschitz assumption of \cite{MR4462746}. Although the time dependence of $\mu(t)$ breaks the usual monotonicity of standard Lyapunov functionals and the reference minimizer of $\varphi$ is generally not a minimizer of $\varphi_{\mu(t)}$, the paper shows that, for $\alpha\ge 3$ and when $\varphi$ attains its minimum, accelerated rates can still be recovered under natural integrability conditions on $\mu$, in agreement with the rates obtained in \cite{MR4462746} for abstract smoothing families. In particular, the  residual satisfies a decay bound of order $\mathcal O(t^{-2})$, and, when $\alpha>3$, one obtains the sharper estimate
			\[
			\varphi_{\mu(t)}(x(t))-\inf \varphi=o(t^{-2}),
			\]
			together with weak convergence of trajectories to a minimizer of $\varphi$ via an Opial-type argument. These accelerated rates improve upon those obtained for the associated first-order dynamics, namely the classical gradient flow driven by the same regularized objective (see Appendix \ref{gradient} for completeness and comparison).
			
			We show that the framework naturally covers max-type scalarizations in multiobjective optimization, where the limit point minimizes the chosen scalarization and is therefore weakly Pareto optimal, as well as finite-scenario distributionally robust optimization. From an applied viewpoint, the results support the idea that acceleration and smoothing can be coupled in continuous time without sacrificing asymptotic consistency with the original nonsmooth problem, provided that the regularization parameter vanishes sufficiently fast.
			
			Several directions emerge for future research. First, it would be of interest to develop \emph{discrete-time} counterparts with explicit complexity guarantees, including accelerated schemes with a nonincreasing regularization schedule $\mu_k$ and step sizes that reflect the $1/\mu_k$ smoothness growth. Second, the present analysis yields weak convergence in Hilbert spaces. Strong convergence follows directly from the value rates when $\varphi$ is strongly convex; in general, it might be enforced by combining the dual regularization considered here with a vanishing Tikhonov term in the primal variable. Third, extending the framework beyond finite-dimensional dual sets $Q\subset\R^m$ (for example to supremum problems with infinite index sets or integral representations arising in stochastic and robust optimization as in \cite{MR4279933}) would require addressing measurability and compactness issues while preserving tractable smooth surrogates. Finally, incorporating constraints on $x$ (via projections or differential inclusions), nonsmooth composite structure in the $g_i$, or stochastic perturbations in the dynamics are natural avenues that could broaden the algorithmic scope of the penalty-based approach.
			
			Overall, the proposed penalty-driven viewpoint provides a flexible bridge between the variational structure of supremum functions and the accelerated dynamics literature, and it offers a principled template for designing and analyzing accelerated methods for nonsmooth convex models with a supremum structure.

			\paragraph{Acknowledgements}
			J.J. Maul\'en was supported by ANID/FONDECYT/Postdoctorado 3250609. E. Vilches was supported by ANID/FONDECYT/Regular 1220886, ANID/FONDECYT/Regular 1240120, and ANID/FONDECYT/Regular 1261728; by ANID/PIA/Basal Funds for Centers of Excellence FB210005; and by the projects ECOS-ANID ECOS230027 and Math-AmSud AMSUD250023. 
			S. Adly is supported by the ECOS Project C24E06 and the FMJH Gaspard Monge Program for Optimization and Data Science.

		\bibliographystyle{abbrv}
		\bibliography{ref}
		
		\appendix
		\renewcommand{\thesection}{\Alph{section}}
		
		\section{Appendix}

		\subsection{Examples}\label{Ejemplos-1000}
		\begin{example}\label{ex:simplex} Let $Q=\Delta_m$, the $m$-dimensional simplex in $\R^m$. Then 
			\[
			\varphi(x)=\max_{i =1,\ldots,m} g_i(x).
			\]
			Different smooth approximations $\varphi_\mu$ are obtained depending on the choice of the penalty $\mathcal{D}$.
			\begin{itemize}
				\item \textbf{Entropic (log-sum-exp) smoothing}
				Let $\mathcal{D}(\lambda)=\operatorname{KL}(\lambda\|u):=\sum_{i=1}^m \lambda_i\log(m\lambda_i)$ be the Kullback-Leibler divergence with respect to the uniform distribution $u=(1/m,\ldots,1/m)\in\Delta_m$. Then
				\[
				\varphi_\mu(x)
				=\mu\log\!\left(\frac{1}{m}\sum_{i=1}^m \exp\!\left(\frac{g_i(x)}{\mu}\right)\right),
				\]
				and one may take $\mathcal{C}=\log m$. This is the classical \texttt{logsumexp} approximation of the maximum. As a simple illustration, consider $\varphi(x)=\max\{x^2+1,\exp(x)\}$. Figure~\ref{fig:ex_logsum} displays $\varphi$
				together with $\varphi_\mu$ for several values of $\mu$. As noted in Remark~\ref{r:minimums}, although $\varphi$ has the
				unique minimizer $x^{\ast}=0$, the minimizer of $\varphi_\mu$ is generally shifted (here, it occurs at a smaller value of $x$).
				
				\item \textbf{Quadratic proximal smoothing} 
				Let $\mathcal{D}(\lambda)=\frac12\|\lambda-u\|_2^2$ with some $u\in\Delta_m$. Then
				\[
				\varphi_\mu(x)
				=\langle g(x),u\rangle
				+\frac{1}{2\mu}\|g(x)\|_2^2
				-\frac{\mu}{2}\,\operatorname{dist}_{\Delta_m}^2\!\left(u+\frac{g(x)}{\mu}\right),
				\]
				and one may take $\mathcal{C}:=\frac12\bigl(1+\|u\|_2^2-2\min_i u_i\bigr)$. In particular, for the uniform choice
				$u_i=1/m$ one has $\mathcal{C}=\frac12\left(1-\frac{1}{m}\right)$. Figure~\ref{fig:ex_simplex} compares
				$\varphi(x)=\max\{x^2+1,\exp(x)\}$ with $\varphi_\mu$ for several values of $\mu$ in the case $u=(1/2,1/2)$.
			\end{itemize}
			\begin{figure}[h]
				\centering
				\subfigure[$\varphi$ and $\varphi_{\mu}$ for entropic (log-sum-exp) smoothing]{\includegraphics[width=0.48\textwidth]{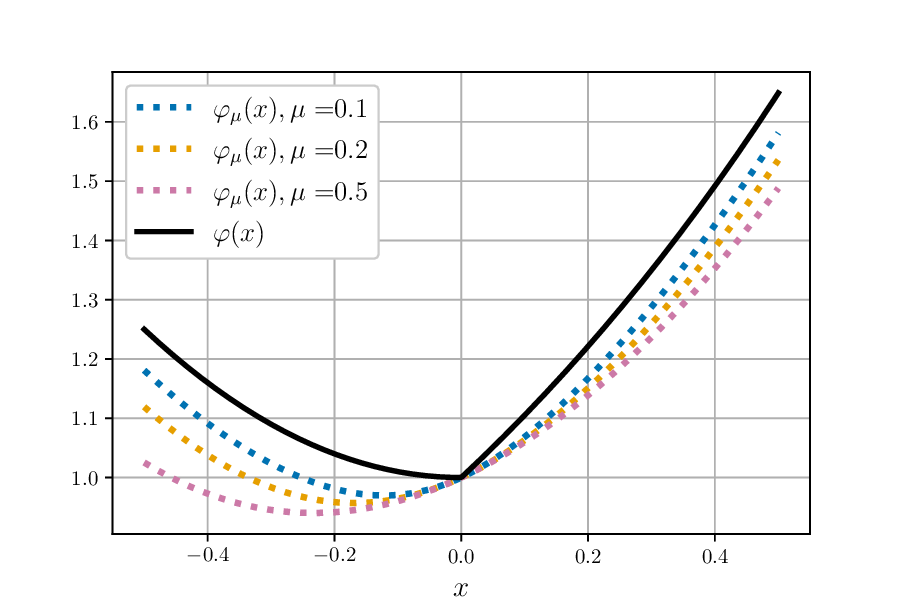} \label{fig:ex_logsum}}
				\subfigure[$\varphi$ and $\varphi_{\mu}$ for quadratic proximal smoothing]{\includegraphics[width=0.48\textwidth]{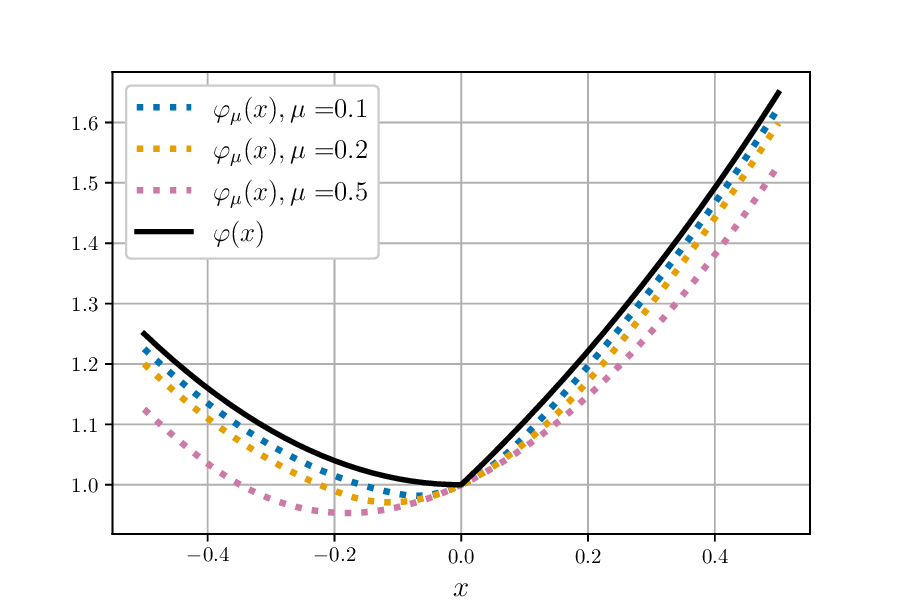}\label{fig:ex_simplex}}
				\caption{Entropic and  proximal smoothing for $g_1(x)=x^2+1$ and $g_2(x)=\exp(x)$ in Example \ref{ex:simplex}.}
				%\label{fig:example1}
			\end{figure}
		\end{example}

		\begin{example}\label{ex:box} Let $\ell, u \in \R^m$ and consider the box 
			$$
			Q=[\ell,u]:=\{\lambda\in \mathbb{R}^m \colon \ell_i\leq \lambda_i\leq u_i, \, i=1,\ldots,m\},  \textrm{ with } 0 \leq \ell_i\leq u_i, \quad i= 1,\ldots,m.
			$$ Then, for $g(x)=(g_1(x),\ldots,g_m(x))\in \mathbb{R}^m$, 
			\[\varphi(x)=\sup_{\lambda \in Q} \langle \lambda, g(x)\rangle =\sum_{i=1}^m \left(u_i\max\{g_i(x),0\}+\ell_i \min\{g_i(x),0\}\right).\]
			Fix any $c\in [\ell,u]$ and set $\mathcal{D}(\lambda)=\frac{1}{2}\Vert \lambda-c\Vert^2_2$. The corresponding regularization admits the closed form
			$$\varphi_{\mu}(x):=\langle g(x),c\rangle+\frac{1}{2\mu}\Vert g(x)\Vert^2_2-\frac{\mu}{2}\operatorname{dist}_{[\ell,u]}^2\left(c+\frac{g(x)}{\mu}\right). 
			$$
			Moreover, one may take $\mathcal{C}:=\frac{1}{2}\sum_{i=1}^m \max\{(\ell_i-c_i)^2,(u_i-c_i)^2\}$.  Figure \ref{fig:ex_box} depicts, for an illustrative instance, the original function $\varphi$ and its regularization $\varphi_{\mu}$ for different values of $\mu$. 
			\begin{figure}[h]
				\centering
				\includegraphics[width=0.65\linewidth]{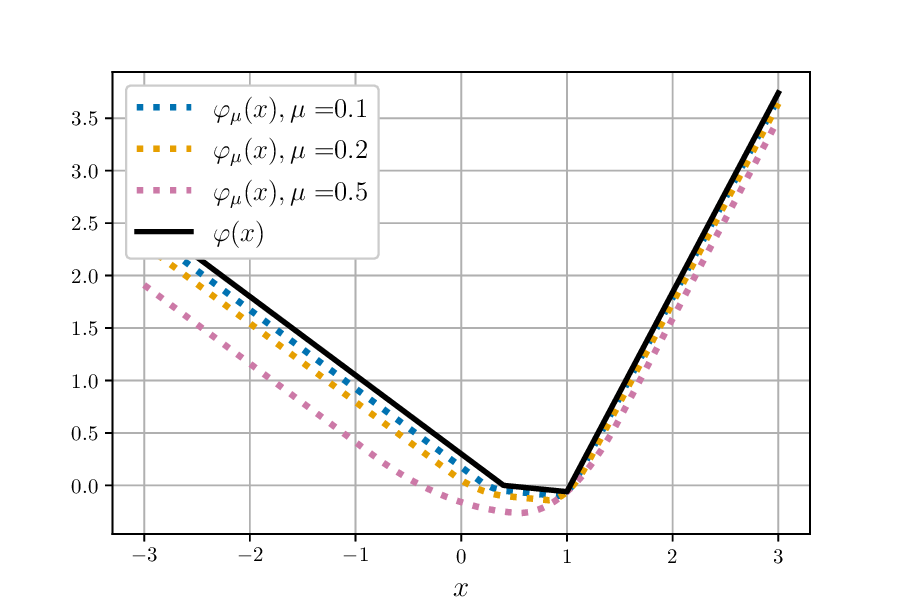}
				\caption{$\varphi$ and $\varphi_{\mu}$ in Example \ref{ex:box} for $g_1(x) = x -1$,  $g_2(x) = -\frac{1}{2}x+\frac{1}{5}$, $Q=[0,2]\times [\frac{1}{5},\frac{3}{2}]$ and $c=(1,\frac{1}{4})$.}
				\label{fig:ex_box}
			\end{figure}
		\end{example}

		\begin{example}
			Let $1\le p\le \infty$ and set 
			\[Q=B_p \cap \R^m_+:=\{\lambda\in\mathbb{R}_{+}^m:\ \|\lambda\|_p\le 1\}.\]
			Let $q$ be the dual exponent, with $q=\infty$ when $p=1$ and $q=1$ when $p=\infty$, and write $g_{+}(x):=(\max\{g_i(x),0\})_{i=1}^m$. Duality of the $p$ and $q$ norms gives
			\[\varphi(x)=\max_{\lambda \in Q} \left\langle \lambda,g(x)\right\rangle=\|g_{+}(x)\|_q.\]
			Choosing the quadratic penalty $\mathcal{D}(\lambda)=\tfrac12\|\lambda\|_2^2$, completing the square yields
			\[\varphi_\mu(x)
			:=\frac{1}{2\mu}\|g(x)\|_2^2-\frac{\mu}{2}\,\operatorname{dist}_{Q}^2\!\left(\frac{g(x)}{\mu}\right).\]
			The penalty is nonnegative, continuous, and strongly convex on $Q$, with infimum equal to zero. One may take
			\[
			\mathcal{C}:=\frac12
			\begin{cases}
				1, & 1\le p\le 2,\\[2mm]
				m^{\,1-\frac{2}{p}}, & 2\le p\le \infty.
			\end{cases}
			\]
			Since $Q \subset \R_{+}^m$, this example belongs to the convex supremum class considered above.
		
		\end{example}
		
		\begin{example}
			Consider $Q=\operatorname{conv}\{a_1,\ldots,a_K\}\subset \mathbb{R}^m_{+}$, where $a_k\in \mathbb{R}^m_{+}$, and set $A=[a_1 \cdots a_K]\in \mathbb{R}^{m\times K}$. Every $\lambda \in Q$ can be written as $\lambda =A\beta$ for some $\beta\in \Delta_K:=\{\beta\in \mathbb{R}^K_+: \sum_{k=1}^K \beta_k=1\}$.   Fix a prior $v\in \operatorname{ri}(\Delta_K)$ and define the pushforward Kullback–Leibler divergence on $Q$ by
			$$
			\mathcal{D}(\lambda):=\min\{ \operatorname{KL}(\beta\Vert v)\colon \beta \in \Delta_K, A\beta=\lambda\}, \quad  \operatorname{KL}(\beta\Vert v):=\sum_{k=1}^K \beta_k \log\left(\frac{\beta_k}{v_k}\right).
			$$ 
			The minimum is attained, and $\mathcal{D}(Av)=0$. The function $\mathcal{D}$ is lower semicontinuous on $Q$. Moreover, if $\Vert A\Vert >0$, strong convexity of KL$(\cdot\Vert v)$ on $\Delta_K$ implies that $\mathcal{D}$ is $\Vert A\Vert^{-2}$-strongly convex on $Q$. If $A=0$, the set $Q$ is a singleton. Thus the penalty satisfies the assumptions used above. Then, 
			$$
			\varphi(x)=\max_{k=1,\ldots,K} \langle a_k,g(x)\rangle \textrm{ and } \varphi_{\mu}(x):=\mu \log\left(\sum_{k=1}^K v_k \exp\left(\frac{\langle a_k,g(x)\rangle}{\mu}\right)\right).
			$$
			In this setting, one may take $\mathcal{C}:=-\log(\min_k v_k)$. In particular, for the uniform prior $v_k=\frac{1}{K}$,
			$$
			\varphi_{\mu}(x):=\mu \log\left(\frac{1}{K}\sum_{k=1}^K \exp\left(\frac{\langle a_k,g(x)\rangle}{\mu}\right)\right) \textrm{ and } \mathcal{C}=\log K.
			$$
		\end{example}

		\subsection{Proof of Proposition \ref{p:Lipschitz}.}\label{prueba-Lipschitz}
		By Proposition \ref{p:properties},
		\[\nabla \varphi_\mu(x) = \sum_{i=1}^{m}\nabla g_i(x)\lambda^\mu_i(x),\]
		where $\lambda^\mu(x) = (\lambda^\mu_1(x),\ldots,\lambda^\mu_m(x))$ is the unique maximizer of the function $\lambda \mapsto \lambda^\top g(x) - \mu \mathcal{D}(\lambda)$ over $Q$. Hence, for any $x$, $y \in B$,  
		\[\Vert \nabla\varphi_\mu(x) - \nabla\varphi_\mu(y) \Vert  \leq \left\Vert \sum_{i=1}^{m}\left( \nabla g_i(x) - \nabla g_i(y) \right)\lambda^\mu_i(x) \right\Vert + \left\Vert \sum_{i=1}^{m}\nabla g_i(y)\left( \lambda^\mu_i(x) - \lambda^\mu_i(y) \right)\right\Vert.\]
		We bound these two terms separately. First, using the triangle inequality and the Lipschitz continuity of $\nabla g_i$ on $B$,
		\[\left\Vert \sum_{i=1}^{m}\left( \nabla g_i(x) - \nabla g_i(y) \right)\lambda^\mu_i(x) \right\Vert \leq \sum_{i=1}^{m}\left\Vert \nabla g_i(x) - \nabla g_i(y)\right\Vert \vert \lambda_i^\mu(x) \vert \leq M_Q L_g \Vert x - y \Vert.\]
		For the second term, we adopt the representation introduced in Remark \ref{r:perspective}, that is, let $\phi=\mathcal{D} + \iota_Q$. Since $\mathcal{D}$ is lower semicontinuous and $\sigma$-strongly convex on the closed convex set $Q$, the extended-valued function $\phi$ is proper, lower semicontinuous, and $\sigma$-strongly convex. Hence $\partial \phi$ is $\sigma$-strongly monotone. Optimality for \eqref{eq:def_lambda_mu} is equivalent to 
		\[g(x) \in \mu\partial \phi \left( \lambda^{\mu}(x) \right).\]
		For $x$, $y \in B$, strong monotonicity applied to the previous gives
		\[\sigma \Vert \lambda^{\mu}(x) - \lambda^{\mu}(y) \Vert^2 \leq \left\langle \dfrac{g(x)-g(y)}{\mu}, \lambda^{\mu}(x) - \lambda^{\mu}(y) \right\rangle.\]
		Therefore, 
		\[\Vert \lambda^{\mu}(x) - \lambda^{\mu}(y) \Vert \leq \frac{1}{\sigma\mu}\Vert g(x) - g(y)\Vert.\] 
		
		It remains to bound $\Vert g(x) - g(y)\Vert_2$. For each $i$, using the fundamental theorem of calculus along the segment $y+t(x-y)\subset B$ (since $B$ is convex), 
		\[\vert g_i(x) - g_i(y) \vert = \left\vert \int_{0}^1 \left\langle \nabla g_i(y + t(x-y)), x-y\right\rangle \, dt\right\vert \leq \sup_{z \in B} \Vert\nabla g_i(z) \Vert_2 \Vert x-y \Vert_2. \]
		Thus,
		\[\Vert g(x) - g(y) \Vert_2 \leq \left(\sum_{i=1}^{m}\sup_{z \in B} \Vert \nabla g_i(z)\Vert^2\right)^{1/2} \Vert x - y\Vert=G_B\Vert x-y\Vert,\]
		and consequently,
		\[ \Vert  \lambda^\mu(x)- \lambda^\mu(y)\Vert_2 \leq \dfrac{G_B}{\mu\sigma }\Vert x-y \Vert.\]
		Finally, 
		\begin{align*}
			\left\Vert \sum_{i=1}^{m}\nabla g_i(y)\left( \lambda^\mu_i(x) - \lambda^\mu_i(y) \right)\right\Vert \leq \Vert \lambda^\mu(x)- \lambda^\mu(y)\Vert_2 \left(\sum_{i=1}^{m}\Vert \nabla g_i(y)\Vert^2\right)^{1/2} 
			\leq\Vert  \lambda^\mu(x)- \lambda^\mu(y)\Vert_2 G_B,
		\end{align*}
		so that
		$$
		\left\Vert \sum_{i=1}^{m}\nabla g_i(y)\left( \lambda^\mu_i(x) - \lambda^\mu_i(y) \right)\right\Vert  \leq \frac{G_B^2}{\mu \sigma}\Vert x-y\Vert.
		$$
		Combining both bounds yields the claim. \qed
		
		\subsection{Proof of Theorem \ref{teo:global_existence}}\label{thm3.1}
		To prove Theorem~\ref{teo:global_existence}, we proceed as follows. On a finite time interval, the regularized gradient defines a continuous vector  field that is locally Lipschitz
		in the state. We first obtain a maximal solution and then use the energy estimate to continue it to the end
		of the interval.
		
		\begin{proposition}\label{prop:local-wp}
			Assume that Hypothesis \ref{hip:general} holds. Fix $\alpha > 0$,  $t_0>0$ and $T>t_0$ with $T<+\infty$. Assume that $\mu \colon [t_0,T]\to(0,+\infty)$ is continuous.  Then, for any initial conditions $x(t_0)=x_0\in\mathcal H$ and $\dot x(t_0)=y_0\in\mathcal H$, there exists a unique maximal solution $x\in C^{2}\big([t_0,\tau_{\max});\mathcal H\big)$ to  \eqref{eq:inertial}, defined on $[t_0,\tau_{\max})$, where  $\tau_{\max}\in (t_0,T]$.
		\end{proposition}
		\begin{proof} Fix $t_0>0$ and $T>t_0$ with $T<+\infty$. Setting $y:=\dot{x}$, the system \eqref{eq:inertial} is equivalent to 
			\begin{equation}\label{sistema-lineal}
				\frac{d}{dt}\binom{x}{y}=
				\binom{y}{-\frac{\alpha}{t}y-\nabla \varphi_{\mu(t)}(x)}
			\end{equation}
			with initial conditions $x(t_0)=x_0$, $y(t_0)=y_0$. Set $\mu_-:=\min_{t\in[t_0,T]}\mu(t)>0$. On every closed ball $B\subset\mathcal H$, Proposition~\ref{p:Lipschitz} gives a Lipschitz constant for $\nabla\varphi_{\mu(t)}$ that is independent of $t\in[t_0,T]$, since $\mu(t)\geq\mu_-$. The map $(t,x)\mapsto\nabla\varphi_{\mu(t)}(x)$ is also continuous. Indeed, if $t_n\to t$ and $x_n\to x$, every cluster point of $\lambda^{\mu(t_n)}(x_n)$ maximizes the limiting dual objective: this follows from compactness of $Q$, lower semicontinuity of $\mathcal D$, and comparison with each fixed $\lambda\in Q$. Uniqueness then gives $\lambda^{\mu(t_n)}(x_n)\to\lambda^{\mu(t)}(x)$. The gradient formula in Proposition~\ref{p:properties} and continuity of each $\nabla g_i$ give the claim. Thus the right-hand side of \eqref{sistema-lineal} is continuous in time and locally Lipschitz in $(x,y)$, uniformly for $t\in[t_0,T]$. The Picard-Lindel\"of theorem on $\mathcal H\times\mathcal H$ yields a unique maximal solution on $[t_0,\tau_{\max})$, where $t_0<\tau_{\max}\leq T$. Continuity of the vector field gives $x\in C^2$ on this interval.
		\end{proof}

		Fix $T>t_0$ and denote by $x_T(t)$ the maximal solution of \eqref{eq:inertial} over the interval $[t_0,\tau_{\max})$ given by Proposition~\ref{prop:local-wp}. Consider the energy functional 
		\[W_T(t):= \frac{1}{2}\Vert \dot{x}_T(t)\Vert^2 + \varphi_{\mu(t)}(x_T(t)) + \mathcal{C}\mu(t),\quad t \in [t_0,\tau_{\max}).\]
		\begin{Lemma} Consider the hypotheses of Proposition~\ref{prop:local-wp} and that $\mu$ is continuously differentiable on $[t_0,T]$. The function $W_T$ is continuously differentiable on $[t_0,\tau_{\max})$ and 
			\begin{equation}\label{eq:W_deriv}
				\dot{W}_T(t) = -\frac{\alpha}{t}\Vert \dot{x}_T(t)\Vert^2 + \dot{\mu}(t)\left( \mathcal{C} - \mathcal{D}(\lambda^{\mu(t)}(x_T(t)))\right)
			\end{equation}
			Consequently, for every $t \in [t_0,\tau_{\max})$, 
			\begin{equation}\label{eq:W_upper}
				W_T(t) \leq W_T(t_0) + \mathcal{C}\int_{t_0}^{T}\vert \dot{\mu}(s)\vert\, ds. 
			\end{equation}
			In addition, if $\mu$ is nonincreasing, then $W_T$ is nonincreasing. 
		\end{Lemma}
		\begin{proof}
			Since $x_T\in C^2$,  \eqref{eq:inertial} gives
			\begin{align*}
				\dot{W}_T(t)&=\Big\langle\dot x_T(t),\ddot x_T(t)+\nabla\varphi_{\mu(t)}(x_T(t))\Big\rangle-\dot\mu(t)\mathcal D\big(\lambda^{\mu(t)}(x_T(t))\big)+\mathcal C\dot{\mu}(t) \\
				&=  -\frac{\alpha}{t}\Vert \dot{x}_T(t)\Vert^2 + \dot{\mu}(t)\left( \mathcal{C} - \mathcal{D}(\lambda^{\mu(t)}(x_T(t)))\right),
			\end{align*}
			and we obtain \eqref{eq:W_deriv}. Since $0\le\mathcal D\le\mathcal C$ on $Q$, we have $0\le\mathcal C-\mathcal D(\lambda^{\mu(t)}(x_T(t)))\le\mathcal C$. Hence $\dot W_T(t)\le\mathcal C|\dot\mu(t)|$, and integration yields \eqref{eq:W_upper}, with the integral being finite because $\dot\mu$ is continuous on $[t_0,T]$. If $\mu$ is nonincreasing, then $\dot\mu\le0$ and \eqref{eq:W_deriv} gives $\dot W_T\le-\frac{\alpha}{t}\Vert\dot x_T\Vert^2\le0$.
		\end{proof}
		
			The bounds just obtained prevent the maximal solution from ending before $T$.
			\begin{Lemma}\label{Lemmacota} Consider the hypotheses of Proposition~\ref{prop:local-wp} and that $\mu$ is continuously differentiable on $[t_0,T]$. Then, there exist constants $C_T$, $Y_T > 0$ such that 
				\[\Vert x_T(t) - x_0\Vert \leq C_T, \quad \Vert \dot{x}_T(t)\Vert \leq Y_T, \quad \forall t \in [t_0,\tau_{\max}).\]
			\end{Lemma}
			\begin{proof} Using \eqref{eq:prop1} and the convexity of $\varphi_{\mu}$, we have
				\[\varphi_{\mu(t)}(x_T(t)) + \mathcal{C}\mu(t) \geq \varphi(x_T(t)) \geq \varphi_{\mu_0}(x_T(t)) \geq \varphi_{\mu_0}(x_0) + \left\langle \nabla \varphi_{\mu_0}(x_0), x_T(t) - x_0\right\rangle,\]
				where $\mu_0=\mu(t_0)$. Let $z_T(t):= \Vert x_T(t) - x_0\Vert$. Using \eqref{eq:W_upper}, we have 
				\[\frac{1}{2}\Vert \dot{x}_T(t)\Vert^2 + \varphi_{\mu_0}(x_0) - \Vert \nabla \varphi_{\mu_0}(x_0)\Vert z_T(t) \leq W_T(t) \leq W_T(t_0) + \mathcal{C}\int_{t_0}^{T}\vert \dot{\mu}(s)\vert\, ds.\]
				That is, we have obtained
				\[\frac{1}{2}\Vert \dot{x}_T(t)\Vert^2    \leq \vartheta + \rho z_T(t),  \]
				for some nonegative constants $\vartheta$, $\rho$, which do not depend on $t$. Then, using the inequality $\sqrt{a} \leq \frac{1+a}{2}$, for all $a \geq 0$,we have 
				\[z_T(t) \leq \int_{t_0}^t \Vert \dot{x}_T (s)\Vert\, ds \leq \int_{t_0}^t \sqrt{2 \vartheta + 2\rho z_T(s)}\, ds \leq \left( \frac{1}{2} + \vartheta\right)(T - t_0) +  \rho  \int_{t_0}^t z_T(s)\, ds.\]
				Since $z_T$ is continuous, we can use Gronwall's lemma to obtain 
				\[z_T(t) \leq C_T:= \left( \frac{1}{2} + \vartheta\right)(T-t_0) \exp\left( \rho  (T-t_0)\right).\]
				We complete the proof by using the previous bound and $\Vert \dot{x}_T(t)\Vert \leq \sqrt{2\vartheta +2\rho z_T(t)}$.
			\end{proof}
			
			Notice that the constants $C_T$ and $Y_T$ depend on the initial conditions, $T$, $\mu$ and $\mathcal{C}$. Then, the previous holds for every $\alpha >0$ and requires neither the existence of minimizer, nor the monotonicity of $\mu$, nor $\mu(t) \to 0$.  
			
			They do not depend on $\alpha$, the monotonicity of $\mu$, or the existence of the minimizers of $\varphi$. 
			\begin{Lemma} Consider the hypotheses of Proposition~\ref{prop:local-wp} and that $\mu$ is continuously differentiable on $[t_0,T]$. Then, the maximal solution $x_T(t)$  of \eqref{eq:inertial} can be defined over all the interval $[t_0,T]$.
			\end{Lemma}
			\begin{proof}
				Let $[t_0,\tau_{\max})$ be the maximal interval from Proposition~\ref{prop:local-wp}, and set $y_T=\dot x_T$. Lemma~\ref{Lemmacota} gives constants $R,Y>0$ such that $\Vert x_T(t)\Vert\leq R$ and $\Vert y_T(t)\Vert\leq Y$ for $t<\tau_{\max}$. Set $B:=\{x\in\mathcal H:\Vert x\Vert\leq R+1\}$. Since each $\nabla g_i$ is Lipschitz on $B$, it is bounded there. The gradient formula and compactness of $Q$ therefore give a constant $K_B$ such that
				$$
				\Vert\nabla\varphi_{\mu(t)}(x)\Vert\leq K_B
				\qquad(t\in[t_0,T],\ x\in B).
				$$
				Along the trajectory, the first-order system \eqref{sistema-lineal} consequently satisfies
				$$
				\Vert\dot x_T(t)\Vert\leq Y,
				\qquad
				\Vert\dot y_T(t)\Vert\leq\frac{\alpha}{t_0}Y+K_B.
				$$
				Thus $(x_T(t),y_T(t))$ is Lipschitz in time on $[t_0,\tau_{\max})$ and has a limit $(\bar x,\bar y)\in\mathcal H\times\mathcal H$ as $t\uparrow\tau_{\max}$. If $\tau_{\max}<T$, the local existence result applied at $(\tau_{\max},\bar x,\bar y)$ extends the solution beyond $\tau_{\max}$, a contradiction. Hence $\tau_{\max}=T$. Defining $(x_T(T),y_T(T))=(\bar x,\bar y)$ gives a solution on $[t_0,T]$. Continuity of the vector field in \eqref{sistema-lineal} yields $x_T\in C^2([t_0,T];\mathcal H)$.
				
			\end{proof}
			Since $T> t_0$ is arbitrary, the solutions on nested finite intervals coincide by uniqueness. Their union is the unique global $C^2$ solution of Theorem \ref{teo:global_existence}.

			\subsection{First-Order Dynamics}\label{gradient}

			Using the previous notations, a similar analysis for the first-order gradient flow can be carried out. Consider the system
			\begin{equation}\label{eq:sist_first}
				\dot{x}(t)  = -\nabla \varphi_{\mu(t)}(x(t)),
			\end{equation}
			with $t\geq t_0 > 0$ and initial condition $x(t_0) = x_0$. On each finite interval $[t_0,T]$, positivity and continuity of $\mu$ give a uniform lower bound on $\mu$. Proposition \ref{p:Lipschitz} then yields local existence and uniqueness. Through this section, assume $\argmin\varphi \neq \emptyset$.  For $x^* \in \argmin \varphi$, convexity and \eqref{eq:prop1} give $\frac{1}{2}\frac{d}{dt}\Vert x(t) - x^*\Vert^2 \leq \mathcal{C}\mu(t)$. Hence $x$ remains bounded on $[t_0,T]$. The gradient is bounded on the resulting ball, so the solution extends to $T$. Since $T$ is arbitrary, the first-order flow has a unique global solution. Its decreasing shifted value gives the following comparison with the inertial estimates.

			\begin{theorem} Assume that Hypothesis \ref{hip:general} holds. Let $\mu\colon [t_0,+\infty)\to (0,\infty)$ be continuously differentiable and nonincreasing, and let $x(\cdot)$ be a solution of \eqref{eq:sist_first}. Set $d_0:=\operatorname{dist}(x(t_0),\operatorname{argmin}\varphi)$. Then the following hold.
				\begin{enumerate}[label={\rm (\roman*)}]
					\item The value gap satisfies
					\[
					-\mathcal{C}\mu(t)\leq \varphi_{\mu(t)}(x(t)) - \inf\varphi \leq \frac{d_0^2}{2(t-t_0)}+ \frac{\mathcal{C}}{(t-t_0)}\int_{t_0}^t \mu(s)ds \quad \textrm{ for all } t> t_0. \]
					Consequently, 
					\[\vert \varphi_{\mu(t)}(x(t)) - \inf\varphi \vert = \mathcal{O}\left( \dfrac{1}{t} + \mu(t) +\frac{1}{t}\int_{t_0}^t \mu(s)ds \right). \]
					\item If $\mu \in L^1(t_0,\infty)$, then 
					\[\vert \varphi_{\mu(t)}(x(t)) - \inf\varphi \vert = \mathcal{O}\left( \dfrac{1}{t}\right).\]
					\item The (unsmoothed) value gap satisfies
					\[\varphi(x(t)) - \inf\varphi \leq \frac{d_0^2}{2(t-t_0)}+\mathcal{C}\mu(t)+\frac{\mathcal{C}}{t-t_0} \int_{t_0}^{t} \mu(s)ds \quad \textrm{ for all } t> t_0.\]
					In particular, if  $\mu \in L^1(t_0,\infty)$, then 
					\[\varphi(x(t)) - \inf\varphi = \mathcal{O}\left( \dfrac{1}{t} \right). \]
				\end{enumerate}
				
			\end{theorem}
			\begin{proof} (i): Let $x^{\ast}\in \operatorname{argmin}\varphi$ arbitrary and observe that 
				\begin{align*}
					\dfrac{d}{dt}\varphi_{\mu(t)}(x(t)) &= \dot{\mu}(t)\dfrac{d}{d\mu}\varphi_{\mu(t)}(x(t)) +\langle \nabla \varphi_{\mu(t)}(x(t)),\dot{x}(t)\rangle \\
					&= -\dot{\mu}(t)\mathcal{D}\left(\lambda^{\mu(t)}(x(t))\right) - \Vert \nabla \varphi_{\mu(t)}(x(t))\Vert^2 \\
					&\leq -\mathcal{C}\dot{\mu}(t) - \Vert \nabla \varphi_{\mu(t)}(x(t))\Vert^2.
				\end{align*}
				Let us define 
				\[\mathcal{F}(t):=  \varphi_{\mu(t)}(x(t)) - \inf\varphi + \mathcal{C}\mu(t).\]
				Inequality \eqref{eq:prop1} makes $\mathcal{F}(t)$ nonnegative, and the preceding derivative estimate makes it nonincreasing. Let us recall that \eqref{eq:prop1} implies $\varphi_{\mu(t)}(x^*)\le \varphi(x^*)=\inf\varphi$. Combining the previous with the convexity of $\varphi_{\mu(t)}$ yields
				\begin{align*}
					\dfrac{d}{dt}\left( \dfrac{1}{2}\Vert x(t) - x^*\Vert^2\right) &= \left\langle \dot{x}(t), x(t) - x^* \right\rangle \\
					&= - \left\langle \nabla \varphi_{\mu(t)}(x(t)), x(t) - x^* \right\rangle \\
					&\leq - \left( \varphi_{\mu(t)}(x(t)) - \varphi_{\mu(t)}(x^*) \right) \\
					&\leq - \left( \varphi_{\mu(t)}(x(t)) - \inf\varphi \right) \\
					&=- \left( \mathcal{F}(t) - \mathcal{C}\mu(t) \right).
				\end{align*}
				Integrating the previous over $[t_0,t]$, rearranging and neglecting the negative term, we obtain 
				\[\int_{t_0}^{t} \mathcal{F}(s)\, ds \leq \frac{1}{2}\Vert x_0 - x^*\Vert^2  + \mathcal{C}\int_{t_0}^{t} \mu(s)\, ds.\]
				Since $\mathcal{F}$ is nonincreasing, one has 
				\[ (t-t_0)\mathcal F(t) \leq \int_{t_0}^{t} \mathcal{F}(s)\, ds \leq \frac{1}{2}\Vert x_0 - x^*\Vert^2  + \mathcal{C}\int_{t_0}^{t} \mu(s)\, ds.\]
				Therefore, for all $t>t_0$
				\[-\mathcal{C}\mu(t)\leq \varphi_{\mu(t)}(x(t)) - \inf\varphi \leq \dfrac{1}{2(t-t_0)}\Vert x_0 - x^*\Vert^2 + \dfrac{\mathcal{C}}{t-t_0}\int_{t_0}^{t} \mu(s)\, ds.\]
				Since $x^{\ast}\in \operatorname{argmin}\varphi$ is arbitrary, we obtain the first inequality of (i). The second assertion of (i) follows directly from the above inequality.\\
				(ii): Since $\mu$ is nonincreasing, $(t-t_0)\mu(t) \leq \int_{t_0}^t \mu(s)\, ds$. Integrability therefore gives $\mu(t) = \mathcal{O}(1/t)$ and $\frac{1}{t-t_0}\int_{t_0}^t\mu(s)\,ds = \mathcal{O}(1/t)$. Hence, the result follows from the first inequality in (i). \\
				(iii): It follows from assertion (i) and  bound \eqref{eq:prop1}. 
			\end{proof}
			Compared with the inertial system \eqref{eq:inertial}, the gradient flow \eqref{eq:sist_first} 
			only requires $\mu\in L^1(t_0,\infty)$ to reach the rate $\mathcal{O}(1/t)$, whereas the 
			accelerated rate $\mathcal{O}(1/t^2)$ requires the stronger condition $t\mu(t)\in L^1(t_0,\infty)$. 
			In other words, acceleration demands a faster vanishing of the smoothing parameter, so that 
			the approximation error remains negligible with respect to the faster decay of the values.
			
			\subsection{Opial's Lemma}\label{opial_Appendix}
			
			In this appendix, we recall Opial's lemma \cite{opial1967weak}, which is a standard tool to establish weak convergence of trajectories in Hilbert spaces and is used in the proof of Theorem~\ref{t:traj}.
			
			\begin{Lemma}[Opial]\label{l:opial} Let $S$ be a nonempty subset of a Hilbert space $\mathcal{H}$, and let $x \colon [t_0,+\infty )\to \mathcal{H}$ be a function. Assume that 
				\begin{enumerate}[label={\rm (\roman*)}]
					\item For every $z \in S$, the limit $\lim_{t \to +\infty} \Vert x(t) - z \Vert$ exists.
					\item Every weak sequential cluster point of $x(\cdot)$ belongs to $S$.
				\end{enumerate}
				Then there exists $x_\infty\in S$ such that $x(t)\rightharpoonup x_\infty$ as $t\to+\infty$.
			\end{Lemma}

		\end{document}